\documentclass[11pt]{article}
\usepackage{mathrsfs}
\usepackage{mathrsfs}
\usepackage{mathrsfs}
\usepackage{diagbox}
\usepackage{cite}
\usepackage{amsfonts}
\usepackage{amssymb}
\usepackage{amsmath}
\usepackage{amsthm}
\usepackage{dsfont}
\usepackage[colorlinks,
            linkcolor=blue,
            anchorcolor=blue,
            citecolor=red
            ]{hyperref}
\theoremstyle{plain}
\newtheorem{thm}{Theorem}[section]

\newtheorem{lem}[thm]{Lemma}
\newtheorem{prop}[thm]{Proposition}
\theoremstyle{definition}
\newtheorem{defn}{Definition}[section]

\newtheorem{Assump}{Assumption}[section]
\theoremstyle{remark}
\newtheorem{rem}{Remark}[section]

\begin{document}
\title{{\Large\bf {
 Optimal control of quantum system in fermion fields: Pontryagin-type maximum principle $\textrm{(I)}$}}}
\author{{\normalsize  Penghui Wang, Shan Wang{\thanks{E-mail addresses: phwang@sdu.edu.cn(P.Wang), 202020244@mail.sdu.edu.cn(S.Wang).}}  \,\,\,\,  } \\
{ School of Mathematics, Shandong University,} {\normalsize Jinan, 250100, China} }
\date{}
\maketitle
\begin{minipage}{14cm} {\bf Abstract }
In this paper, the Pontryagin-type maximum principle for optimal control of quantum stochastic systems in fermion fields is obtained.  These systems have gained significant prominence in numerous quantum applications ranging from physical chemistry to multi-dimensional nuclear magnetic resonance experiments.    
Furthermore, we establish the existence and uniqueness of solutions to backward quantum stochastic differential equations driven by fermion Brownian motion.   
The application of noncommutative martingale inequalities and the martingale representation theorem enables this achievement.

\end{minipage}
 \maketitle
\numberwithin{equation}{section}
\newtheorem{theorem}{Theorem}[section]
\newtheorem{lemma}[theorem]{Lemma}
\newtheorem{proposition}[theorem]{Proposition}
\newtheorem{corollary}[theorem]{Corollary}
\newtheorem{remark}[theorem]{Remark}


\section{Introduction}\label{Intro}
\indent\indent

Let $\mathscr{H}$ be a separable complex Hilbert space. The anti-symmetric Fock space over $\mathscr{H}$ \cite{B.S.W.1,P.Book} is defined by
\begin{equation*}
\Lambda(\mathscr{H}):=\bigoplus_{n=0}^\infty\Lambda_n(\mathscr{H}),
\end{equation*}
where $\Lambda_n(\mathscr{H})$ is the Hilbert space anti-symmetric $n$-fold tensor product of $\mathscr{H}$ with itself, and $\Lambda_0(\mathscr{H}):=\mathbb{C}$. For any $z\in\mathscr{H}$, the creation operator $C(z):\Lambda_n(\mathscr{H})\rightarrow \Lambda_{n+1}(\mathscr{H})$ defined by $v\mapsto \sqrt{n+1}\ z\wedge v$, is a bounded operator on $\Lambda(\mathscr{H})$ with $\|C(z)\|=\|z\|$. 
The annihilation operator $A(z)$ is the adjoint of $C(z)$, i.e. $A(z)=C(z)^*$.
The fermion field $\Psi(z)$ is defined on $\Lambda(\mathscr{H})$ by
\begin{equation}\label{definition Psi}
\Psi(z):=C(z)+A(Jz),
\end{equation}
where the map $J: \mathscr{H}\rightarrow\mathscr{H}$ is a conjugation operator (i.e., $J$ is antilinear, antiunitary, and $J^2=1$). 
 The canonical anti-commutation relation holds:
\begin{equation}\label{CAR-Psi}
\{\Psi(z),\Psi(z')\}\equiv\Psi(z)\Psi(z')+\Psi(z')\Psi(z)=2\langle Jz',z\rangle I,\quad z,z'\in \mathscr{H}.
\end{equation}
Denote by $\mathscr{C}$ the von Neumann algebra generated by the bounded operators $\{\Psi(z): z\in \mathscr{H}\}$. 
For the Fock vacuum $\Omega\in\Lambda(\mathscr{H})$, define
\begin{equation*}
m(\cdot):=\langle\Omega, \cdot\Omega\rangle_{\Lambda(\mathscr{H})}
\end{equation*}
on $\mathscr{C}$.  By \cite{B.S.W.1,E.B,James.2012,P.Book}, $m(\cdot)$ is a faithful, normal, central state on $\mathscr{C}$, which is also called the quantum expectation with respect to the Fock vacuum and is denoted by $\mathbb{E}(\cdot)$. The space $(\mathscr{C}, m)$ is a quantum (noncommutative) probability space. 
For  $p\in[1,\infty)$, we define the noncommutative $L^p$-norm on $\mathscr{C}$ by
 $$\|f\|_{p}:=m\left(|f|^p\right)^\frac{1}{p}=\left\langle\Omega, |f|^p\Omega\right\rangle_{\Lambda(\mathscr{H})}^\frac{1}{p},$$
where $|f|=(f^*f)^\frac{1}{2}$. The space $L^p(\mathscr{C},m)$ is the completion of $(\mathscr{C}, \|\cdot\|_p)$, which is the noncommutative $L^p$-space, abbreviated as $L^p(\mathscr{C})$.

Quantum (Noncommutative) probability theory has attracted significant attention since it was recognized as a new branch of mathematics.
In particular, quantum stochastic calculus has also received attention with various degrees of completeness in the setting of the bosonic and fermionic Fock space and for Clifford algebras in the works of Segal\cite{S.}, Barnett, Streater and Wilde \cite{B.S.W.1,B.S.W.2,B.S.W.3,S-1}, Applebaum and Hudson \cite{A.,A.H-1,A.H-2,H.L.}, Parthasarathy \cite{P.Book,P.K.R.,P.H}, Belavkin \cite{B}, Gordina \cite{G}, Sinha and Goswami \cite{S.G}.
Among them, Haudson, Lindsay, Barnett, Streater and Wilde investigated the solutions to quantum stochastic differential equations (QSDEs for short).
Belavkin \cite{B} employed quantum stochastic methods that were developed in the 1980s to describe quantum noise and a quantum generalization of the It\^{o} calculus.
Later, Gough, Guta, James and Nurdin \cite{G.G.J.N} developed fermion filtering theory using the fermion quantum stochastic calculus.

In what follows, let $\mathscr{H}=L^2(\mathbb R^+)$, and $Jf=\bar{f}$ for $f\in L^2(\mathbb{R}^+)$. Let $\mathscr{C}$ be von Neumann algebra generated by $\{\Psi(v): v\in L^2(\mathbb{R}^+)\}$, and
$\{\mathscr{C}_t\}_{t\geq 0}$ be an increasing family of von Neumann subalgebras of $\mathscr{C}$ generated by $\{\Psi(v): v\in L^2(\mathbb{R}^+)\ \textrm{and}\ \textrm{ess supp}\ v\subset [0,t]\}$.
The family of von Neumann subalgebras $\{\mathscr{C}_t\}_{t\geq 0}$ is called the filtration of $\mathscr{C}$ \cite{P.X}.
The fermion Brownian motion $W(t)$ is defined by
\begin{equation*}
W(t):=\Psi(\chi_{[0,t]})=C(\chi_{[0,t]})+A(J\chi_{[0,t]}),\quad t\geq 0.
\end{equation*}
Obviously, $W(t)$ is self-adjoint.
In this paper, we consider the controlled QSDE driven by the fermion Brownian motion $W(t)$ in $L^p(\mathscr{C})$:
\begin{equation}\label{FSDE}
 \left\{
 \begin{aligned}
  dx(t)=&D(t,x(t),u(t))dt+F(t,x(t),u(t))dW(t)+dW(t)G(t,x(t),u(t)),\ \textrm{in}\ (0,T], \\
 x(0)=&x_0,
  \end{aligned}
  \right.
\end{equation}
where $x_0\in L^p(\mathscr{C}_{t_0})$ is the initial condition,  $u(\cdot)\in \mathcal{U}(0,T)$ is the control variable. 
Here, the control domain $\mathcal{U}(0,T)$ is defined by
\begin{equation}\label{the admissible control domain}
  \mathcal{U}(0,T):=\left\{u:[0,T]\to U;\ u(\cdot)\ \textrm{is}\ \{\mathscr{C}_{t}\}_{t\geq0}\textrm{-adapted, continuous} \right\},
\end{equation}
where $U$ is a linear subspace of  $L^p(\mathscr{C})$. The control domain $U$ is also a metric space with the metric $d(u_1,u_2)=\|u_1-u_2\|_p$.
As in \cite[(S1), Page 388]{L.Z-2020}, the maps $D(\cdot,\cdot,\cdot), F(\cdot,\cdot,\cdot), G(\cdot,\cdot,\cdot):[0,T]\times L^p(\mathscr{C})\times U\to L^p(\mathscr{C})$ are adapted, where the definition of the adapted maps is provided in Definition \ref{the definition of the adapted maps}. 
In \eqref{FSDE}, $u(\cdot)$ is called the control, while $x(\cdot)=x(\cdot;x_0,u(\cdot))$ is the corresponding state process. It should be noted that in the framework of quantum stochastic calculus, based on the noncommutativity of operators, both the right integral $\int_0^tf(s)dW(s)$ and the left integral $\int_0^tdW(s)f(s)$ exist.

Let $x(\cdot)$ be the solution to \eqref{FSDE} corresponding to the control $u(\cdot)$. 
Define the cost functional $\mathcal{J}(\cdot)$ as follows:
\begin{equation}\label{Cost functional introduced}
  \mathcal{J}(u(\cdot)):=\int_{0}^TL(t,x(t),u(t))dt+h(x(T)),\quad  u(\cdot)\in\mathcal{U}(0,T),
\end{equation}
where the maps $L(\cdot,\cdot,\cdot):[0,T]\times L^p(\mathscr{C})\times U\to \mathbb{R}$ and $h(\cdot):L^p(\mathscr{C}_T)\to \mathbb{R}$. In particular, the maps $\widehat{L}(\cdot,\cdot,\cdot):[0,T]\times L^p(\mathscr{C})\times U\to L^p(\mathscr{C})_{sa}$ and $\widehat{h}(\cdot):L^p(\mathscr{C}_T)\to L^p(\mathscr{C}_T)_{sa}$, and elements of $L^p(\mathscr{C})_{sa}$ are observables.  And for any $u(\cdot)\in\mathcal{U}(0,T)$,  the cost functional 
\begin{equation*}
\mathcal{J}(u(\cdot))=\mathbb{E}\left(\int_{0}^T\widehat{L}(t,x(t),u(t))dt+\widehat{h}(x(T))\right)
\end{equation*}
was considered by\cite{James.2005,James.2012,S.M}.

Similar to \cite{G.B.S,James.2012,M-M.M}, we consider quantum optimal control problem for \eqref{FSDE}:\\
\textbf{Problem(QOC)}.
Find a control $\bar{u}(\cdot)\in\mathcal{U}(0,T)$ such that
\begin{equation}\label{OP}
\mathcal{J}(\bar{u}(\cdot))=\inf_{u(\cdot)\in\mathcal{U}(0,T)}\mathcal{J}(u(\cdot)).
\end{equation}
Any $\bar{u}(\cdot)\in\mathcal{U}(0,T)$ satisfying \eqref{OP} is called an optimal control. The corresponding $\bar{x}(\cdot)$ and $(\bar{x}(\cdot), \bar{u}(\cdot))$ are called an optimal state process and optimal pair of quantum control systems, respectively.

Optimal control theory is a powerful mathematical tool, which has been rapidly developed since the 1950s, mainly for engineering applications. Recently, this method has become widely used to
improve process performance in quantum technologies by means of the highly efficient control of quantum dynamics \cite{W.M,E.W}.
Pontryagin's maximum principle and Bellman's dynamic programming principle are two of the most important tools for solving optimal control problems.
Belavkin, Smolyanov, James et al. \cite{E.B,B.H.J,B.N.M,C.B,E.W,G.B.S,G.B.S-06,G.B.S-06Bellman,James.2005,M-M.M,G.G.J.N,S.M} tackled the optimal control problem  by quantum stochastic calculus and dynamic programming methods\textit{ in the case that the control domain $U$ is of finite dimension}.
More precisely, Belavkin, Gough and Smolyanov \cite{G.B.S-06,G.B.S-06Bellman} investigated the quantum optimal control problem, which is the evolution of a quantum system subject to continuous measurements
governed by the QSDE
\begin{equation}\label{QSDE-GBS}
  dX_t=\omega(t,u_t,X_t)dt+\sum_{\alpha=1}^m\sigma_\alpha(t,u_t,X_t)dM^\alpha_t,
\end{equation}
where $\{X_t\}$ denotes a stochastic process, $\{M^\alpha; \alpha=1,\cdots, m\}$ denotes a sequence of martingales. 
In this case, the optimal cost is defined by
\begin{equation*}
\mathcal{J}[\{u\};t_0,X_0]=\int_{t_0}^T\textit{l}(s,u_s,X_s)ds+g(X_T),
\end{equation*}
where $\{X_s: s\in [t_0,T]\}$ is the solution to \eqref{QSDE-GBS}  with the initial condition $X_0$, $u$ denotes a continuous control function, taking values in $\mathbb{R}^n$. 
James \cite{James.2005},  Sharifi and Momeni \cite{S.M} investigated the optimal control problems of the following  quantum systems
\begin{equation*}
dX(t)=\omega(t,X(t),u(t))dt+\sigma(t,X(t))dw(t),
\end{equation*}
where the maps $\omega$ and $\sigma$ are Lipschitz continuous with respect to $X(t)$, $w(t)$ is a martingale.
They derived Bellman equation and identified the optimal control $u$ with the cost functional $\mathcal{J}[\{u\};t_0,X_0]$, such that 
$$S(t_0,X_0)=\inf_{u}{\mathcal{J}}[\{u\};t_0,X_0]=\inf_{u\in\mathcal{U}(t)}\mathbb{E}\left\{\int_{t_0}^TC(u(t),X(t))dt+G(u(T),X(T))\right\},$$
where $S(t_0,X_0)$ is the value function, the cost density $C(u(t),X(t))$ and the terminal cost $G(u(T),X(T))$ are observables, and the control domain $\mathcal{U}(t)$ is a subspace of $\mathbb{R}^n$. 
The adaptedness of the control variable $u$ is not considered.
And then,
Mulero-Mart\'{i}nez and Molina-Vilaplana \cite{M-M.M} derived the quantum Pontryagin maximum principle in a global form from the Hamilton-Jacobi-Bellman equation for quantum optimal control with adapted control variable $u$.
Boscain, Sigalotti and Sugny \cite{B.S.S} described modern aspects of optimal control theory, with a particular focus on the Pontryagin maximum principle, 
where the finite-dimensional quantum control system is considered. 
The aforementioned results are about quantum optimal control problems with finite-dimensional control domains.
However, quantum optimal control problems with infinite-dimensional control domain have rarely been investigated.

In classical probability theory, there is a great deal of research on optimal control theory. Since the 1970s, the maximum principle has been extensively studied for stochastic control systems. Peng, Yong, Zhang et al.  obtained corresponding results on the Pontryagin-type maximum principle in \cite{F.Z-2020,H.P,L.Z-2013,L.Z-2014,L.Z-2015,L.Z-2020,D.M,F.H.T,Y.Z}. Peng \cite{P} investigated optimal control problems for control systems when the control enters the diffusion term and the control domain $U$ is nonconvex. L\"{u} and Zhang \cite{L.Z-2013,L.Z-2014,L.Z-2015} derived the necessary conditions for optimal controls under the case of convex and nonconvex control domains, respectively.
Furthermore, Du and Meng \cite{D.M}, Fuhrman, Hu and Tessitore \cite{F.H.T},  L\"{u} and Zhang \cite{L.Z-2014} were concerned with the general stochastic maximum principle of infinite dimensions.

Inspired by the classical infinite-dimensional optimal control theory, this paper investigates the optimal control problem of \textit{infinite-dimensional quantum systems with the infinite-dimensional control domains}.
To obtain the Pontryagin-type maximum principle, the following main ideas play key roles.
\begin{itemize}
  \item 
  For $p\in[1,\infty)$, the noncommutative space $L^p(\mathscr{C})$ can be decomposed into
  \begin{equation*}
   L^p(\mathscr{C})=L^p(\mathscr{C}_e)\oplus L^p(\mathscr{C}_o),
  \end{equation*}
  and the details are listed in \cite[Proposition 3.3]{P.X}. The fermion Brownian motion $\{W(t); t\geq 0\}$ commutes with the elements of $L^p(\mathscr{C}_e)$, and anti-commutes with the elements of $L^p(\mathscr{C}_o)$. This allows us to overcome the difficulty of the noncommutativity of QSDEs in fermion fields.
  \item To address quantum stochastic calculus with respect to the fermion Brownian motion $\{W(t); t\geq 0\}$, the canonical anti-commutation relation \eqref{CAR-Psi} is critical, which implies that
\begin{equation*}
 W(t)^*W(t)=W(t)^2=t \textrm{I}, \quad t\in[0,\infty).
\end{equation*}
  \item The Burkholder-Gundy inequality with respect to noncommutative martingales, given by Pisier and Xu \cite{P.X}, plays a significant role in solving the QSDEs in fermion fields and obtaining the corresponding estimates. Based on this, we prove that the solution to the corresponding backward quantum stochastic differential equation (BQSDE for short) is a necessary condition of \textbf{Problem (QOC)} with an infinite-dimensional control domain $U$. This is the Pontryagin-type maximum principle of quantum control systems.
   \item Finally, we investigate the solution to BQSDEs by using the noncommutative martingales representation theorem and noncommutative martingale inequalities.
\end{itemize}

The paper is organized as follows. In  section \ref{Pre}, we review some basic notations on fermion fields and main lemmas.
Section \ref{Principle} formulates the Pontryagin-type maximum principle for quantum optimal control problems in $L^p(\mathscr{C})$. 
In section \ref{solution-QBSDE}, we prove the existence and uniqueness of the solution to BQSDEs in noncommutative space.


\section{Preliminaries}\label{Pre}
\indent\indent
In this section, we introduce notations based on references \cite{B.S.W.1,B.S.W.2,B.S.W.3,P.X,D.P.B,M.T,W.F} and some lemmata, which will be used later. 

\begin{defn}\cite{B.S.W.1,B.S.W.2,P.X}\label{the definition of the adapted maps}
A map $x:\mathbb{R}^+\rightarrow L^p(\mathscr{C})$ is said to be adapted if $x(t)\in L^p(\mathscr{C}_t)$ for each $t\in\mathbb{R}^+$. A map $F:\mathbb{R}^+\times L^p(\mathscr{C})\rightarrow L^p(\mathscr{C})$ is said to be adapted if $F(t, u)\in L^p(\mathscr{C}_t)$ for any $t\in \mathbb{R}^+$ and $u\in L^p(\mathscr{C}_t)$.
\end{defn}

Throughout this paper, we let $T >0$ be a fixed time horizon and denote
\begin{align*}
 & C_\mathbb{A}(0,T; L^p(\mathscr{C})):=\left\{ f:[0,T]\to L^p(\mathscr{C})\ |\ f(\cdot)\  {\rm is}\ \{\mathscr{C}_t\}_{t\geq0}{\rm -adapted\ and\  continuous}  \right\},
\end{align*}
with the following norm
\begin{equation*}
 \|f(\cdot)\|_{C_\mathbb{A}(0,T; L^p(\mathscr{C}))}:=\sup_{t\in[0,T]}\|f(t)\|_p.
\end{equation*}
It is clear that $C_\mathbb{A}(0,T; L^p(\mathscr{C}))$ is a Banach space.
As usual, for any Banach space $X_1$ and $X_2$, $\mathcal{L}(X_1;X_2)$ is the Banach space of all bounded linear operators from $X_1$ to $X_2$.
For $p_1,p_2\in (1,\infty)$, put
\begin{equation*}
\begin{aligned}
&L^\infty_\mathbb{A}(0,T;\mathcal{L}(L^{p_1}(\mathscr{C});L^{p_2}(\mathscr{C})))\\
&:=\Big\{T:[0,T]\to \mathcal{L}(L^{p_1}(\mathscr{C});L^{p_2}(\mathscr{C})); T(\cdot)\ \textrm{is measurable and essentially bounded}, \\
&\indent\indent\indent\indent\indent\indent\indent\textrm{and}\ T(t)\in \mathcal{L}(L^{p_1}(\mathscr{C}_t);L^{p_2}(\mathscr{C}_t))\ \textrm{a.e. on }[0,T] \Big\}.
\end{aligned}
\end{equation*}
Recall that in \cite{P.X} the grading automorphism $\Upsilon$ on $L^p(\mathscr{C})$ is uniquely determined by
\begin{equation*}
\Upsilon(\Psi(v_1)\Psi(v_2)\cdots \Psi(v_n))=(-1)^n\Psi(v_1)\Psi(v_2)\cdots \Psi(v_n),\quad v_i\in \mathscr{H},\ 1\leq i\leq n.
\end{equation*}
\begin{defn}\cite{P.X,B.S.W.1,B.S.W.2}
An element $f \in L^p(\mathscr{C})$ is said to be even (resp.
odd) if $\Upsilon(f) = f$ (resp. $\Upsilon(f)=-f$).
\end{defn}
\begin{lem}\cite{X.C.B,P.X}\label{BGp and Minkowski}
Let $p\in[1,\infty)$. For $f\in \mathcal{H}^p(0,T)$, its It\^{o}-Clifford integral
$\int_0^tf(s)dW(s)$ and $\int_0^tdW(s)f(s)$
are $L^p$-martingales for $t \in[0,T]$. Moreover,
it holds that
\begin{equation}\label{Burkholder-Gundy inequality}
\left\|\int_0^tdW(s)f(s)\right\|_p\simeq_p\left\|\int_0^tf(s)dW(s)\right\|_p\simeq_p \|f\|_{\mathcal{H}^p(0,t)},\   t\in[0, T].
\end{equation}
If $ p\in (1,2]$, then
\begin{equation}\label{Minkowski of p in (1,2]}
\left(\int_0^t\|f(s)\|_p^2ds\right)^{\frac{1}{2}}\leq\mathcal{C}_p\left\|\int_0^tf(s)dW(s)\right\|_p.
\end{equation}
If $ p\in [2,\infty)$, then
\begin{equation}\label{Minkowski of p in [2,n)}
\left\|\int_0^tf(s)dW(s)\right\|_p\leq \mathcal{C}_p\left(\int_0^t\|f(s)\|_p^2ds\right)^{\frac{1}{2}}.
\end{equation}
\end{lem}

\section{The Pontryagin-type Maximum Principle}\label{Principle}
\indent\indent
This section is devoted to obtaining the Pontryagin-type maximum principle for quantum optimal control of \eqref{FSDE} in $L^p(\mathscr{C})$ for $p\in[2,\infty)$.
For the  state  equation \eqref{FSDE} and the cost functional \eqref{Cost functional introduced}, we impose the following assumptions.
\begin{Assump}\label{Assump 1}
\begin{description}
  \item[(A1)] The maps $D(\cdot,\cdot,\cdot), F(\cdot,\cdot,\cdot), G(\cdot,\cdot,\cdot):[0,T]\times L^p(\mathscr{C})\times U\to L^p(\mathscr{C})$ are adapted, and there exists a constant $\mathcal{C}>0$ such that for any $(t,u)\in[0, T]\times U$, $x, \hat{x}\in L^p(\mathscr{C})$,
\begin{equation*}
\left\{
\begin{aligned}
&\|D(t,x,u)-D(t,\hat{x},u)\|_p\leq \mathcal{C}\|x-\hat{x}\|_p,\\
&\|F(t,x,u)-F(t,\hat{x},u)\|_p\leq \mathcal{C}\|x-\hat{x}\|_p,\\
&\|G(t,x,u)-G(t,\hat{x},u)\|_p\leq \mathcal{C}\|x-\hat{x}\|_p,\\
&\|D(t,0,u)\|_p+\|F(t,0,u)\|_p+\|G(t,0,u)\|_p\leq \mathcal{C}.
\end{aligned}
\right.
\end{equation*}
  \item[(A2)] The maps $L:[0,T]\times L^p(\mathscr{C})\times U\to \mathbb{R}$ and $h:L^p(\mathscr{C}_T)\to \mathbb{R}$ are measurable, and there exists a constant $\mathcal{C}>0$ 
   such that for any $(t, u)\in[0, T]\times U$, $x, \hat{x}\in L^p(\mathscr{C})$,
      \begin{equation*}
\left\{
\begin{aligned}
&|L(t,x,u)-L(t,\hat{x},u)|\leq \mathcal{C}\|x-\hat{x}\|_p,\\
&|h(x)-h(\hat{x})|\leq \mathcal{C}\|x-\hat{x}\|_p,\\
&|L(t,0,u)|+|h(0)|\leq \mathcal{C}.
\end{aligned}
\right.
\end{equation*}
  \item[(A3)]The maps $D,\ F,\  G,\ L$ and $h$ are second order Fr\'{e}chet differentiable on $L^p(\mathscr{C})$.
For any $(t,u)\in[0,T]\times U$,  the maps
$D_x(t,\cdot,u),\ F_{x}(t,\cdot,u),\  G_{x}(t,\cdot,u):L^p(\mathscr{C})\to \mathcal{L}(L^p(\mathscr{C}))$
and
$L_x(t,\cdot,u),\ h_x(\cdot): L^p(\mathscr{C})\to L^{p'}(\mathscr{C})$  
 are continuous, the maps $ D_{xx}(t,\cdot,u),\ F_{xx}(t,\cdot,u)$, $G_{xx}(t,\cdot,u): L^p(\mathscr{C})\to \mathcal{L}(L^p(\mathscr{C}),L^p(\mathscr{C});L^p(\mathscr{C}))$
and
 $L_{xx}(t,\cdot ,u),\ h_{xx}(\cdot): L^p(\mathscr{C})\to \mathcal{L}(L^p(\mathscr{C});L^{p'}(\mathscr{C}))$ are continuous. 
 Moreover, for any $(t,x, u)\in [0,T]\times L^p(\mathscr{C})\times U$,
   \begin{equation*}
\left\{
\begin{aligned}
&\|D_{x}(t,x,u)\|_{\mathcal{L}(L^{p}(\mathscr{C}))}+\|F_{x}(t,x,u)\|_{\mathcal{L}(L^{p}(\mathscr{C}))}+\|G_{x}(t,x,u)\|_{\mathcal{L}(L^{p}(\mathscr{C}))}\leq \mathcal{C},\\
&\|D_{xx}(t,x,u)\|_{\mathcal{L}(L^p(\mathscr{C}),L^p(\mathscr{C}); L^{p'}(\mathscr{C}))}+\|F_{xx}(t,x,u)\|_{\mathcal{L}(L^p(\mathscr{C}),L^p(\mathscr{C}); L^{p'}(\mathscr{C}))}\\
&\indent\indent+\|G_{xx}(t,x,u)\|_{\mathcal{L}(L^p(\mathscr{C}),L^p(\mathscr{C}); L^{p'}(\mathscr{C}))}\leq \mathcal{C},\\
&\|L_{x}(t,x,u)\|_{p'}+\|h_{x}(x)\|_{p'}\leq \mathcal{C},\\
&\|L_{xx}(t,x,u)\|_{\mathcal{L}(L^p(\mathscr{C}); L^{p'}(\mathscr{C}))}+\|h_{xx}(x)\|_{\mathcal{L}(L^p(\mathscr{C}); L^{p'}(\mathscr{C}))}\leq \mathcal{C}.
\end{aligned}
\right.
\end{equation*}
\end{description}
\end{Assump}
\noindent In what follows, when there is no confusion, denote $p'$ the conjugate number of $p$, i.e. $\frac{1}{p}+\frac{1}{p'}=1$.

Before stating the main result, we introduce the following auxiliary result.
\begin{lem}\label{the estimate of x}
Under the above conditions, there exists a unique solution $x(\cdot)\in C_\mathbb{A}(0,T;L^p(\mathscr{C}))$ to \eqref{FSDE} for $p\in[2,\infty)$. Moreover,
\begin{equation}\label{estimate of solution w.r.t. intial condition}
  \sup_{t\in[0,T]}\|x(t)\|_p^2\leq \mathcal{C}\left(1+\|x_0\|_p^2\right).
\end{equation}
\end{lem}
\begin{proof}
By \cite[Theorem 3.1]{J.W.W}, there exists a unique solution $x(\cdot)\in C_\mathbb{A}(0,T;L^p(\mathscr{C}))$ to \eqref{FSDE}, and 
$\{x(t)\}_{t\geq 0}$ satisfies the integral equation
\begin{align*}
x(t)=x_0+\int_0^tD(\tau, x(\tau),u(\tau))d\tau+\int_0^tF(\tau, x(\tau),u(\tau))dW(\tau)+\int_0^tdW(\tau)G(\tau, x(\tau),u(\tau)).
\end{align*}
By Assumption \ref{Assump 1} \textbf{(A1)}, we have
\begin{equation}\label{D^e}
\begin{aligned}
  \|D(t, x(t),u(t))\|_p
  \leq&\|D(t, x(t),u(t))-D(t, 0,u(t))\|_p+\|D(t,0,u(t))\|_p\\
  \leq&\mathcal{C}(\|x(t)\|_p+1), \ {\rm a.e.}\ t\in[0,T].
  \end{aligned}
\end{equation}
Similarly, for $t\in[0,T]$,
\begin{equation}\label{F^e, G^e}
 \|F(t, x(t),u(t))\|_p\leq\mathcal{C}(\|x(t)\|_p+1)\quad {\rm and}\quad  \|G(t, x(t),u(t))\|_p\leq\mathcal{C}(\|x(t)\|_p+1).
\end{equation}
Therefore, by the Minkowski inequality, the H\"{o}lder inequality, \eqref{Minkowski of p in [2,n)}, \eqref{D^e} and \eqref{F^e, G^e},
\begin{align*}
 \|x(t)\|_p^2
\leq &  4\|x_0\|_p^2+4T\int_0^t\|D(\tau, x(\tau),u(\tau))\|_p^2d\tau\\
   &+4\mathcal{C}_p\int_0^t\|F(\tau, x(\tau),u(\tau))\|_p^2d\tau+4\mathcal{C}_p\int_0^t\|G(\tau, x(\tau),u(\tau))\|_p^2d\tau\\
   \leq &  4\|x_0\|_p^2+4\mathcal{C}_{p,D,F,G,T}\int_0^t(1+\|x(t)\|_p^2)d\tau.
\end{align*}
By the Gronwall inequality and the Lebesgue Dominated Convergence Theorem, we conclude that
\begin{equation*}
\sup_{t\in[0,T]}\|x(t)\|_p^2\leq \mathcal{C}(\|x_0\|_p^2+1).
\end{equation*}
\end{proof}

Let $(\bar{x}(\cdot), \bar{u}(\cdot))$ be the given optimal pair of $\textbf{Problem (QOC)}$ of  \eqref{FSDE} in $L^p(\mathscr{C})$ with the cost functional
\begin{equation}\label{cost function even}
\mathcal{J}(\bar{u}(\cdot))=\int_0^T L(t,\bar{x}(t),\bar{u}(t))dt+h(\bar{x}(T)).
\end{equation}
Then it satisfies the following quantum control system
\begin{equation}\label{FSDE-of-xe in even}
\left\{
\begin{aligned}
d\bar{x}(t)=&D(t, \bar{x}(t),\bar{u}(t))dt +F(t, \bar{x}(t),\bar{u}(t))dW(t)+dW(t)G(t, \bar{x}(t),\bar{u}(t)),\ \rm{in}\ (0,T],\\
\bar{x}(0)=&x_0.
\end{aligned}
\right.
\end{equation}
Fix any $u(\cdot)\in \mathcal{U}(0,T)$, $\varepsilon>0$, we define
\begin{equation}\label{Definition of u}
u^\varepsilon(t):=\left\{
\begin{aligned}
\bar{u}(t), \quad  &t\in[0,T]\backslash E_\varepsilon,\\
u(t),\quad &t\in E_\varepsilon,
\end{aligned}
\right.
\end{equation}
where $E_\varepsilon\subseteq [0,T]$ is a measurable set with $|E_\varepsilon|=\varepsilon$.
Let $x^\varepsilon(\cdot)$ be the state process of \eqref{FSDE} corresponding to the control variable $u^\varepsilon(\cdot)$, that is,
\begin{equation}\label{FSDE-of-xe,ue-in even}
\left\{
\begin{aligned}
dx^\varepsilon(t)=&D(t, x^\varepsilon(t),u^\varepsilon(t))dt+F(t, x^\varepsilon(t),u^\varepsilon(t))dW(t)+dW(t)G(t, x^\varepsilon(t),u^\varepsilon(t)),\quad {\rm in}\ (0,T],\\
x^\varepsilon(0)=&x_0.
\end{aligned}
\right.
\end{equation}

For the sake of convenience, we denote for $\varphi=D, F, G, L$,
\begin{equation}\label{abbreviation}
\left\{
\begin{aligned}
&\varphi_x(t):=\varphi_x(t, \bar{x}(t),\bar{u}(t)),\\
&\varphi_{xx}(t):=\varphi_{xx}(t, \bar{x}(t),\bar{u}(t)),\\
&\delta\varphi(t):= \varphi(t,\bar{x}(t),u(t))- \varphi(t,\bar{x}(t),\bar{u}(t)),\\
&\delta\varphi_x(t):= \varphi_x(t,\bar{x}(t),u(t))- \varphi_x(t,\bar{x}(t),\bar{u}(t)),\\
&\delta\varphi_{xx}(t):= \varphi_{xx}(t,\bar{x}(t),u(t))- \varphi_{xx}(t,\bar{x}(t),\bar{u}(t)),\\
&\widetilde{\varphi}_x(t):=\int_0^1\varphi_x(t,\bar{x}(t)+\theta(x^\varepsilon(t)-\bar{x}(t)),u^\varepsilon(t))d\theta,\\
&\widetilde{\varphi}_{xx}(t):= 2\int_0^1(1-\theta) \varphi_{xx}(t,\bar{x}(t)+\theta (x^\varepsilon(t)-\bar{x}(t)), u^\varepsilon(t))d\theta.
\end{aligned}
\right.
\end{equation}
Let $y^\varepsilon(\cdot)$ and $z^\varepsilon(\cdot)$ be respectively the solution to the following QSDEs:
\begin{equation}\label{FSDE-y-e}
\left\{
\begin{aligned}
dy^\varepsilon(t)=&D_x(t) y^\varepsilon(t)dt+\left\{F_x(t) y^\varepsilon(t)+\delta F(t)\chi_{E_\varepsilon}(t)\right\}dW(t)\\
 & \quad +dW(t)\left\{G_x(t) y^\varepsilon(t)+\delta G(t)\chi_{E_\varepsilon}(t)\right\},\quad  \rm{in}\ (0,T],\\
y^\varepsilon(0)=&0,
\end{aligned}
\right.
\end{equation}
and
\begin{equation}\label{FSDE-z-e}
\left\{
\begin{aligned}
dz^\varepsilon(t)=
&\left\{F_x(t) z^\varepsilon(t)+\delta F_x(t)y^\varepsilon(t)\chi_{E_\varepsilon}(t)+\frac{1}{2}F_{xx}(t)\left(y^\varepsilon(t),y^\varepsilon(t)\right)\right\}dW(t)\\
 & +dW(t)\left\{G_x(t) z^\varepsilon(t)+\delta G_x(t)y^\varepsilon(t)\chi_{E_\varepsilon}(t)+\frac{1}{2}G_{xx}(t)(y^\varepsilon(t),y^\varepsilon(t))\right\}\\
 &+\left\{D_{x}(t) z^\varepsilon(t)+\delta D(t)\chi_{E_\varepsilon}(t)+\frac{1}{2}D_{xx}(t)(y^\varepsilon(t),y^\varepsilon(t))\right\}dt\quad \rm{in}\ (0,T],\\
z^\varepsilon(0)=&0,
\end{aligned}
\right.
\end{equation}
where the maps $D_{xx}(\cdot),\ F_{xx}(\cdot),\ G_{xx}(\cdot)\in L_\mathbb{A}^\infty(0,T;\mathcal{L}(L^p(\mathscr{C}), L^p(\mathscr{C});L^p(\mathscr{C})))$ are introduced in Assumption \ref{Assump 1}. This means that, for any $(t, x,u)\in [0,T]\times L^p(\mathscr{C})\times U$ and $x_1,x_2\in L^p(\mathscr{C})$, $D_{xx}(t,x,u)(x_1,x_2), F_{xx}(t,x,u)(x_1,x_2), G_{xx}(t,x,u)(x_1,x_2)\in L^p(\mathscr{C})$.
\begin{thm}
\label{all estimate and Taylor of cost functional}
Suppose that  Assumption \ref{Assump 1} holds. Then, for $p\in[2,\infty)$, 
\begin{gather}
 \sup_{t\in[0,T]}\|x^\varepsilon(t)-\bar{x}(t)\|_{p}^{2}=\textbf{o}(\varepsilon);\label{estimate-x-p-e}\\
  \sup_{t\in[0,T]}\|y^\varepsilon(t)\|_{p}^{2}=\textbf{o}(\varepsilon);\label{estimate-y-p-e}\\
  \sup_{t\in[0,T]}\|z^\varepsilon(t)\|_{p}^{2}=\textbf{o}(\varepsilon^{2}); \label{estimate-z-p-e}\\
 \sup_{t\in[0,T]}\|x^\varepsilon(t)-\bar{x}(t)-y^\varepsilon(t)\|_{p}^{2}=\textbf{o}(\varepsilon^{2});\label{estimate-x-y-p-e} \\
  \sup_{t\in[0,T]}\|x^\varepsilon(t)-\bar{x}(t)-y^\varepsilon(t)-z^\varepsilon(t)\|_{p}^{2}=\textbf{o}(\varepsilon^{2}).\label{estimate-x-y-z-p-e}
\end{gather}
Moreover, the following expansion holds for the cost functional:
\begin{equation}\label{cost functional-J-even}
\begin{aligned}
  \mathcal{J}(u^\varepsilon(\cdot))=&\mathcal{J}(\bar{u}(\cdot))+{\rm{Re}}\left\langle h_x(\bar{x}(T),\ y^\varepsilon(T)+z^\varepsilon(T)\right\rangle +\frac{1}{2}{\rm{Re}}\left\langle h_{xx}(\bar{x}(T)y^\varepsilon(T),\ y^\varepsilon(T)\right\rangle\\
  &+{\rm{Re}}\int_0^T\Big\{\left\langle L_x(t), y^\varepsilon(t)+z^\varepsilon(t)\right\rangle+\frac{1}{2}\left\langle L_{xx}(t) y^\varepsilon(t), y^\varepsilon(t)\right\rangle +\delta L(t)\chi_{E_\varepsilon}(t)\Big\}dt+\textbf{o}(\varepsilon).
\end{aligned}
\end{equation}\end{thm}
\begin{proof}
Let $\xi^\varepsilon(t):=x^\varepsilon(t)-\bar{x}(t)$ for  $t\in[0,T]$.
Thus,  $\xi^\varepsilon(\cdot)$ satisfies the following QSDE:
\begin{equation}\label{FSDE-T-1-e}
\left\{
\begin{aligned}
d\xi^\varepsilon(t)=&\left\{\widetilde{D}_x(t)\xi^\varepsilon(t)+\delta D(t)\chi_{E_\varepsilon}(t)\right\}dt
+\left\{\widetilde{F}_x(t)\xi^\varepsilon(t)+\delta F(t)\chi_{E_\varepsilon}(t)\right\}dW(t)\\
&+dW(t)\left\{\widetilde{G}_x(t)\xi^\varepsilon(t)+\delta G(t)\chi_{E_\varepsilon}(t)\right\}, \quad  \textrm{in} \ (0,T],\\
\xi^\varepsilon(0)=&0.
\end{aligned}
\right.
\end{equation}
From  Assumption \ref{Assump 1} \textbf{(A3)}, Lemma \ref{BGp and Minkowski} and the H\"{o}lder inequality, we have 
\begin{equation}\label{the first eatimate of xi}
\begin{aligned}
\|\xi^\varepsilon(t)\|_p^2\leq &\mathcal{C} \int_0^t\|\widetilde{F}_x(\tau)\xi^\varepsilon(\tau)+\delta F(\tau)\chi_{E_\varepsilon}(\tau)\|_p^2+\|\widetilde{G}_x(\tau)\xi^\varepsilon(\tau)+\delta G(\tau)\chi_{E_\varepsilon}(\tau)\|_p^2d\tau\\
&+\mathcal{C}\left(\int_0^t\|\widetilde{D}_x(\tau)\xi^\varepsilon(\tau)\|_p+\|\delta D(\tau)\chi_{E_\varepsilon}(\tau)\|_p(\tau)d\tau\right)^2\\
\leq&\mathcal{C}\int_0^t\|\xi^\varepsilon(\tau)\|_p^2d\tau+\mathcal{C}\left(\int_0^t\|\delta D(\tau)\chi_{E_\varepsilon}(\tau)\|_pd\tau\right)^2\\
&+\mathcal{C}\int_0^t\|\delta F(\tau)\chi_{E_\varepsilon}(\tau)\|_p^2+\|\delta G(\tau)\chi_{E_\varepsilon}(\tau)\|_p^2d\tau.
\end{aligned}
\end{equation}
Under  Assumption \ref{Assump 1} \textbf{(A1)}, it follows from \eqref{estimate of solution w.r.t. intial condition} that, for $t\in[0,T]$
\begin{equation}\label{estimate D of xi}
\begin{aligned}
  \|\delta D(t)\|_p
  \leq& \|D(t,\bar{x}(t), u(t))-D(t,0, u(t))\|_p+\|D(t,0, u(t))-D(t,0, \bar{u}(t))\|_p\\
  &+\|D(t,0, \bar{u}(t))-D(t,\bar{x}(t), \bar{u}(t))\|_p\\
  \leq &\mathcal{C}(\|\bar{x}(t)\|_p+1)\\
  \leq &\mathcal{C}(\|x_0\|_p+1).
  \end{aligned}
\end{equation}
Similarly, we have 
\begin{equation}
\|\delta F(t)\|_p \leq \mathcal{C}(\|x_0\|_p+1),\quad \|\delta G(t)\|_p \leq \mathcal{C}(\|x_0\|_p+1).
\end{equation}
From \eqref{estimate of solution w.r.t. intial condition} and \eqref{estimate D of xi}, we obtain
\begin{equation}\label{estimate of integral with respect of D with s}
\begin{aligned}
 \int_0^t\|\delta D(\tau)\chi_{E_\varepsilon}(\tau)\|_pd\tau\leq &\int_0^t\|\delta D(\tau)\|_p\chi_{E_\varepsilon}(\tau)d\tau\\
  \leq &\mathcal{C}\int_0^t(\|x_0\|_p+1)\chi_{E_\varepsilon}(\tau)d\tau.
\end{aligned}
\end{equation}
Similar to the above reasoning, we have
\begin{equation}\label{estimate F G of xi}
\begin{aligned}
 \int_0^t\|\delta F(\tau)\chi_{E_\varepsilon}(\tau)\|_p^2+\|\delta G(\tau)\chi_{E_\varepsilon}(\tau)\|_p^2d\tau
 &\leq \int_0^t\left\{\|\delta F(\tau)\|_p^2+\|\delta G(\tau)\|_p^2\right\}\chi_{E_\varepsilon}(\tau)d\tau\\
 &\leq \mathcal{C} \int_0^t(\|x_0\|_p^2+1)\chi_{E_\varepsilon}(\tau)d\tau.
\end{aligned}
\end{equation}
By the Gronwall inequality, together with \eqref{the first eatimate of xi}, \eqref{estimate of integral with respect of D with s}, \eqref{estimate F G of xi},  we can infer that
\begin{equation}\label{final estimate of xe-x even}
 \sup_{t\in[0,T]}\|\xi^\varepsilon(t)\|_{p}^{2}
 \leq \mathcal{C}(\|x_0\|_p^2+1)\varepsilon.
\end{equation}
This proves \eqref{estimate-x-p-e}.

Now, we provide estimate on $y^\varepsilon$.
From \eqref{FSDE-y-e}, we obtain that
\begin{align*}
  \|y^\varepsilon(t)\|_p
  \leq &\left\|\int_0^tD_x(\tau)y^\varepsilon(\tau)d\tau\right\|_p+\left\|\int_0^tF_x(\tau)y^\varepsilon(\tau)dW(\tau)\right\|_p+\left\|\int_0^t dW(\tau)G_x(\tau)y^\varepsilon(\tau)\right\|_p\\
  &+\left\|\int_0^t dW(\tau)G_x(\tau)y^\varepsilon(\tau)\right\|_p+\left\|\int_0^tdW(\tau)\delta G(\tau)\chi_{E_\varepsilon}(\tau)\right\|_p\\
  \leq &\mathcal{C}\left(\int_0^t\|y^\varepsilon(\tau)\|_p^2d\tau\right)^{\frac{1}{2}}+\mathcal{C}\left(\int_0^t\left(\|x_0\|^2_p+1\right)\chi_{E_\varepsilon}(\tau)d\tau\right)^{\frac{1}{2}},\quad t\in[0,T].
\end{align*}
By means of the Gronwall inequality, we yield that
\begin{equation}\label{y-x0}
\sup\limits_{t\in[0,T]}\|y^\varepsilon(t)\|_p^2\leq \mathcal{C}(\|x_0\|_p^2+1)\varepsilon.
\end{equation}
Hence, \eqref{estimate-y-p-e} holds.

Next, we prove \eqref{estimate-z-p-e}. From \eqref{FSDE-z-e}, for $t\in[0,T]$,
\begin{align*}
  z^\varepsilon(t)= &\int_0^t\left\{D_x(\tau)z^\varepsilon(\tau)+\delta D(\tau)\chi_{E_\varepsilon}(\tau)+\frac{1}{2}D_{xx}(\tau)(y^\varepsilon(\tau),y^\varepsilon(\tau))\right\}d\tau\\
  &+\int_0^t\left\{F_x(\tau)z^\varepsilon(\tau)+\delta F_x(\tau)y^\varepsilon(\tau)\chi_{E_\varepsilon}(\tau)+\frac{1}{2}F_{xx}(\tau)(y^\varepsilon(\tau),y^\varepsilon(\tau))\right\}dW(\tau)\\
   &+\int_0^tdW(\tau)\left\{G_x(\tau)z^\varepsilon(\tau)+\delta G_x(\tau)y^\varepsilon(\tau)\chi_{E_\varepsilon}(\tau)+\frac{1}{2}G_{xx}(\tau)(y^\varepsilon(\tau),y^\varepsilon(\tau))\right\}.
\end{align*}
Based on  Assumption \ref{Assump 1} $\textbf{(A3)}$, we have
\begin{equation}\label{D-xx of even}
\begin{aligned}
 \left\|\int_0^tD_{xx}(\tau)(y^\varepsilon(\tau),y^\varepsilon(\tau))d\tau\right\|_p&\leq\int_0^t\|D_{xx}(\tau)\|_{\mathcal{L}(L^p(\mathscr{C}),L^p(\mathscr{C});L^p(\mathscr{C}))}\|y^\varepsilon(\tau)\|_p^2d\tau\\
 &\leq\mathcal{C}\int_0^t\|y^\varepsilon(\tau)\|_p^2d\tau.
 \end{aligned}
\end{equation}
By  Lemma \ref{BGp and Minkowski}, similar to \eqref{D-xx of even}, we find that
\begin{equation}\label{bouned linear operator of Fxx}
\begin{aligned}
 \left\|\int_0^tF_{xx}(\tau)(y^\varepsilon(\tau),y^\varepsilon(\tau))dW(\tau)\right\|_p
 \leq&\mathcal{C}_p\left(\int_0^t\|F_{xx}(\tau)(y^\varepsilon(\tau),y^\varepsilon(\tau))\|_p^2d\tau\right)^{\frac{1}{2}}\\
\leq&\mathcal{C}_{p,F}\left(\int_0^t\|y^\varepsilon(\tau)\|_p^4d\tau\right)^{\frac{1}{2}},
\end{aligned}
\end{equation}
and
\begin{equation}\label{G-xx of even}
  \left\|\int_0^tG_{xx}(\tau)(y^\varepsilon(\tau),y^\varepsilon(\tau))dW(\tau)\right\|_p\leq\mathcal{C}_{p,G}\left(\int_0^t\|y^\varepsilon(\tau)\|_p^4d\tau\right)^{\frac{1}{2}}.
\end{equation}
Combined with  Lemma \ref{BGp and Minkowski} and \eqref{estimate D of xi}, one has
\begin{equation}\label{estimate of even Feye delta}
 \begin{aligned}
    \left\|\int_0^t\delta F_x(\tau)y^\varepsilon(\tau)\chi_{E_\varepsilon}(\tau)dW(\tau)\right\|_p^2 
   \leq&\mathcal{C}_p\int_0^t\left\|\delta F_x(\tau)y^\varepsilon(\tau)\right\|_p^2\chi_{E_\varepsilon}(\tau)d\tau\\
   \leq & \mathcal{C}_p\int_0^t\|\delta F_x(\tau)\|^2_{\mathcal{L}(L^p(\mathscr{C}))}\|y^\varepsilon(\tau)\|_p^2\chi_{E_\varepsilon}(\tau)d\tau\\
    \leq & \mathcal{C}\int_0^t(\|x_0\|_p^2+1)\|y^\varepsilon(\tau)\|_p^2\chi_{E_\varepsilon}(\tau)d\tau,
 \end{aligned}
\end{equation}
and
\begin{equation}\label{estimate of even Geye delta}
\begin{aligned}
    \left\|\int_0^tdW(\tau)\delta G_x(\tau)y^\varepsilon(\tau)\chi_{E_\varepsilon}(\tau)\right\|_p^2
   &\leq\mathcal{C}_p\int_0^t\left\|\delta G_x(\tau)y^\varepsilon(\tau)\right\|_p^2\chi_{E_\varepsilon}(\tau)d\tau\\
&\leq\mathcal{C}\int_0^t(\|x_0\|_p^2+1)\|y^\varepsilon(\tau)\|_p^2\chi_{E_\varepsilon}(\tau)d\tau.
\end{aligned}
\end{equation}
Hence, it follows from Lemma \ref{BGp and Minkowski}, \eqref{estimate D of xi}, \eqref{estimate F G of xi}, \eqref{y-x0} and \eqref{D-xx of even}-\eqref{estimate of even Geye delta} that
\begin{align*}
\|z^\varepsilon(t)\|_{p}^{2}\leq&\mathcal{C}\left(\int_0^t\|z^\varepsilon(\tau)\|_{p}^{2} d\tau+\int_0^t\|y^\varepsilon(\tau)\|_{p}^{4} d\tau+\left(\int_0^t(\|x_0\|_p+1)\chi_{E_\varepsilon}(\tau)d\tau\right)^2\right.\\
&\left.+\int_0^t(\|x_0\|_p^2+1)\|y^\varepsilon(\tau)\|_{p}^{2}\chi_{E_\varepsilon}(\tau)d\tau\right)\\
\leq &\mathcal{C}\left(\int_0^t\|z^\varepsilon(\tau)\|_{p}^{2} d\tau+(\|x_0\|_p^2+1)\varepsilon^2\right),
\end{align*}
this, together with the Gronwall inequality, implies that
\begin{equation*}
   \sup_{t\in[0,T]}\|z^\varepsilon(t)\|_{p}^{2}
   \leq \mathcal{C}\left(\|x_0\|^{2}_p+1\right)\varepsilon^{2} .
\end{equation*}

In order to prove \eqref{estimate-x-y-p-e}.
we set $\eta^\varepsilon(t):=\xi^\varepsilon(t)-y^\varepsilon(t)=x^\varepsilon(t)-\bar{x}(t)-y^\varepsilon(t)$, for $t\in[0,T].$
By calculation, $\eta^\varepsilon$ solves the following QSDE:
\begin{equation}\label{FSDE-T-3-e}
\left\{
\begin{aligned}
d\eta^\varepsilon(t)=&\left\{\widetilde{D}_x(t)\eta^\varepsilon(t)+\left(\widetilde{D}_x(t)-D_x(t)\right)y^\varepsilon(t)+\delta D(t)\chi_{E_\varepsilon}(t)\right\}dt\\
&+\left\{\widetilde{F}_x(t)\eta^\varepsilon(t)+\left(\widetilde{F}_x(t)-F_x(t)\right)y^\varepsilon(t)\right\}dW(t)\\
&+dW(t)\left\{\widetilde{G}_x(t)\eta^\varepsilon(t)+\left(\widetilde{G}_x(t)-G_x(t)\right)y^\varepsilon(t)\right\}, \quad  \textrm{in}\ (0,T],\\
\eta^\varepsilon(0)=&0.
\end{aligned}
\right.
\end{equation}
Hence,
\begin{equation}\label{estimate-T-2-p-e}
\begin{aligned}
\|\eta^\varepsilon(t)\|_{p}^2&\leq\mathcal{C} \left\{\int_0^t\|\eta^\varepsilon(\tau)\|_p^2+\left\|\left(\widetilde{F}_x(\tau)-F_x(\tau)\right)y^\varepsilon(\tau)\right\|_{p}^2+\left\|\left(\widetilde{G}_x(\tau)-G_x(\tau)\right)y^\varepsilon(\tau)\right\|_{p}^2d\tau\right.\\
&\indent\left.+\left(\int_0^t\left\|\left(\widetilde{D}_x(\tau)-D_x(\tau)\right)y^\varepsilon(\tau)\right\|_{p}+\|\delta D(\tau)\chi_{E_\varepsilon}(\tau)\|_pd\tau\right)^2\right\}.
\end{aligned}
\end{equation}

The terms on the right side of \eqref{estimate-T-2-p-e} will now be analyzed individually. 
Under Assumption \ref{Assump 1} $\textbf{(A3)}$, we have
\begin{equation}\label{Estimate of Dxx}
\begin{aligned}
   &\left\|\widetilde{D}_x(\tau)-D_x(\tau))\right\|_{\mathcal{L}(L^p(\mathscr{C}))}\\
  =& \left\|\int_0^1\left\{D_x(\tau,\bar{x}(\tau)+\theta\xi^\varepsilon(\tau),u^\varepsilon(\tau))-D_x(\tau)\right\}d\theta\right\|_{\mathcal{L}(L^p(\mathscr{C}))}\\
  =&\bigg\|\int_0^1\left\{D_x(\tau,\bar{x}(\tau)+\theta\xi^\varepsilon(\tau),u^\varepsilon(\tau))-D_x(\tau,\bar{x}(\tau), u^\varepsilon(\tau))\right\}\\
  &\indent\indent+\left\{D_x(\tau,\bar{x}(\tau), u^\varepsilon(\tau))-D_x(\tau)\right\}d\theta\bigg\|_{\mathcal{L}(L^p(\mathscr{C}))}\\
  =&\left\|\int_0^1\left\{\int_0^1D_{xx}(\tau,\bar{x}(\tau)+\sigma\theta\xi^\varepsilon(\tau),u^\varepsilon(\tau))\sigma\xi^\varepsilon(\tau)d\sigma+\chi_{E_\varepsilon}(\tau)\delta D_x(\tau) \right\} d\theta\right\|_{\mathcal{L}(L^p(\mathscr{C}))}\\
  \leq&\mathcal{C}(\|\xi^\varepsilon(\tau)\|_p+\chi_{E_\varepsilon}(\tau)), \ \textrm{a.e.}\ \tau\in[0,T].
\end{aligned}
\end{equation}
It follows from  \eqref{estimate-x-p-e}, \eqref{estimate-y-p-e} and \eqref{Estimate of Dxx} that
\begin{equation}\label{Estimate of Dxx and yex}
\begin{aligned}
 \int_0^t\left\|\left(\widetilde{D}_x(\tau)-D_x(\tau)\right)y^\varepsilon(\tau)\right\|_{p}d\tau
 \leq&  \int_0^t\left\|\widetilde{D}_x(\tau)-D_x(\tau))\right\|_{\mathcal{L}(L^p(\mathscr{C}))}\left\|y^\varepsilon(\tau)\right\|_{p}d\tau\\
\leq&\mathcal{C}\int_0^t(\|\xi^\varepsilon(t)\|_p+\chi_{E_\varepsilon}(\tau))\|y^\varepsilon(t)\|_pd\tau\\
\leq&\mathcal{C}\varepsilon.
 \end{aligned}
\end{equation}
Similar to \eqref{Estimate of Dxx}, we have
\begin{align*}
\left\|\widetilde{F}_x(\tau)-F_x(\tau)\right\|_{\mathcal{L}(L^p(\mathscr{C}))} \leq &\mathcal{C}(\|\xi^\varepsilon(\tau)\|_p+\chi_{E_\varepsilon}(\tau)), \ \textrm{a.e.}\ \tau\in[0,T] ,\\
\left\|\widetilde{G}_x(\tau)-G_x(\tau)\right\|_{\mathcal{L}(L^p(\mathscr{C}))} \leq &\mathcal{C}(\|\xi^\varepsilon(\tau)\|_p+\chi_{E_\varepsilon}(\tau)), \ \textrm{a.e.}\ \tau\in[0,T].
\end{align*}
Hence, similar to \eqref{Estimate of Dxx and yex}, we obtain that
\begin{equation}\label{Estimate of Fxx and yex}
\begin{aligned}
    \int_0^T\left\|\left(\widetilde{F}_x(\tau)-F_x(\tau)\right)y^\varepsilon(\tau)\right\|_{p}^2d\tau \leq&\mathcal{C}\int_0^T \left\|\widetilde{F}_x(\tau)-F_x(\tau)\right\|^2_{\mathcal{L}(L^p(\mathscr{C}))} \|y^\varepsilon(\tau)\|_{p}^2d\tau\\
   \leq&\mathcal{C}\varepsilon^2,
\end{aligned}
\end{equation}
\begin{equation}\label{Estimate of Gxx and yex}
\int_0^T\left\|\left(\widetilde{G}_x(\tau)-G_x(\tau)\right)y^\varepsilon(\tau)\right\|_{p}^2d\tau
   \leq\mathcal{C}\varepsilon^2.
\end{equation}
From \eqref{estimate of integral with respect of D with s} \eqref{estimate-T-2-p-e}, \eqref{Estimate of Dxx and yex}-\eqref{Estimate of Gxx and yex},
we conclude that
\begin{equation}\label{finally-xe-xx-y}
 \sup_{t\in[0,T]}\|\eta^\varepsilon(t)\|_p^{2}\leq \mathcal{C}\varepsilon^2.
\end{equation}

Finally, we prove \eqref{estimate-x-y-z-p-e}.
Let
$$\zeta^\varepsilon(t):=x^\varepsilon(t)-\bar{x}(t)-y^\varepsilon(t)-z^\varepsilon(t)=\eta^\varepsilon(t)-z^\varepsilon(t), \quad t\in[0,T].$$
A direct calculation gives
\begin{equation}\label{equation-T-3-e}
\left\{
\begin{aligned}
d\zeta^\varepsilon(t)=&\left\{D_x(t)\zeta^\varepsilon(t)+\Theta_1(t\right)\}dt+\left\{F_x(t)\zeta^\varepsilon(t)+\Theta_2(t)\right\}dW(t)\\
&+dW(t)\left\{G_x(t)\zeta^\varepsilon(t)+\Theta_3(t)\right\}, \quad  \textrm{in}\ (0,T],\\
\zeta^\varepsilon(0)=&0,
\end{aligned}
\right.
\end{equation}
where
\begin{equation*}
\begin{aligned}
\Theta_1(t):=&\delta D_x(t)\xi^\varepsilon(t)\chi_{E_\varepsilon}(t)+\frac{1}{2}\left\{\widetilde{D}_{xx}(t)-D_{xx}(t,\bar{x}(t), u^\varepsilon(t))\right\}(\xi^\varepsilon(t),\xi^\varepsilon(t))\\
&+\frac{1}{2}D_{xx}(t)\left\{(\xi^\varepsilon(t),\xi^\varepsilon(t))-(y^\varepsilon(t),y^\varepsilon(t))\right\}+\frac{1}{2}\delta D_{xx}(t)(\xi^\varepsilon(t),\xi^\varepsilon(t))\chi_{E_\varepsilon}(t),\\
\Theta_2(t):=&\delta F_x(t)\eta^\varepsilon(t)\chi_{E_\varepsilon}(t)+\frac{1}{2}\left\{\widetilde{F}_{xx}(t)-F_{xx}(t,\bar{x}(t), u^\varepsilon(t))\right\}(\xi^\varepsilon(t),\xi^\varepsilon(t))\\
&+\frac{1}{2}F_{xx}(t)\{(\xi^\varepsilon(t),\xi^\varepsilon(t))-(y^\varepsilon(t),y^\varepsilon(t))\}+\frac{1}{2}\delta F_{xx}(t)(\xi^\varepsilon(t),\xi^\varepsilon(t))\chi_{E_\varepsilon}(t),\\
\Theta_3(t):=&\delta G_x(t)\eta^\varepsilon(t)\chi_{E_\varepsilon}(t)+\frac{1}{2}\left\{\widetilde{G}_{xx}(t)-G_{xx}(t,\bar{x}(t), u^\varepsilon(t))\right\}(\xi^\varepsilon(t),\xi^\varepsilon(t))\\
&+\frac{1}{2}G_{xx}(t)\{(\xi^\varepsilon(t),\xi^\varepsilon(t))-(y^\varepsilon(t),y^\varepsilon(t))\}+\frac{1}{2}\delta G_{xx}(t)(\xi^\varepsilon(t),\xi^\varepsilon(t))\chi_{E_\varepsilon}(t).
\end{aligned}
\end{equation*}
Indeed, from \eqref{FSDE-of-xe in even}, \eqref{FSDE-of-xe,ue-in even}-\eqref{FSDE-z-e}, 
it is easy to verify that the drift term for the equation solved by $\zeta^\varepsilon(\cdot)$ is
\begin{equation}\label{drift term of T-3-e}
\begin{aligned}
&D(t,x^\varepsilon(t),u^\varepsilon(t))-D(t,\bar{x}(t),\bar{u}(t))-D_x(t)y^\varepsilon(t)-D_x(t)z^\varepsilon(t)\\
&-\delta D(t)\chi_{E_\varepsilon}(t)-\frac{1}{2}D_{xx}(t)(y^\varepsilon(t),y^\varepsilon(t))\\
=&D(t,x^\varepsilon(t),u^\varepsilon(t))-D(t,\bar{x}(t),u^\varepsilon(t))-D_x(t)\{y^\varepsilon(t)+z^\varepsilon(t)\}-\frac{1}{2}D_{xx}(t)(y^\varepsilon(t),y^\varepsilon(t)).
\end{aligned}
\end{equation}
For $ \theta\in[0,1]$, 
by Taylor's formula with integral type, we can deduce that
\begin{equation}\label{Taylor of Dxe}
\begin{aligned}
   & D(t,x^\varepsilon(t),u^\varepsilon(t))-D(t,\bar{x}(t),u^\varepsilon(t)) \\
  = &D_x(t,\bar{x}(t),u^\varepsilon(t))\xi^\varepsilon(t)+\int_0^1(1-\theta)D_{xx}(t,\bar{x}(t)+\theta \xi^\varepsilon(t),u^\varepsilon(t))(\xi^\varepsilon(t),\xi^\varepsilon(t))d\theta\\
  =&D_x(t,\bar{x}(t),u^\varepsilon(t))\xi^\varepsilon(t)+\frac{1}{2}\widetilde{D}_{xx}(t)(\xi^\varepsilon(t),\xi^\varepsilon(t)).
\end{aligned}
\end{equation}
Next,
\begin{equation}\label{First Frechet of D}
  \begin{aligned}
  &D_x(t,\bar{x}(t),u^\varepsilon(t))\xi^\varepsilon(t)-D_x(t)(y^\varepsilon(t)+z^\varepsilon(t))\\
  =&\{D_x(t,\bar{x}(t),u^\varepsilon(t))-D_x(t)\}\xi^\varepsilon(t)+D_x(t)\{\xi^\varepsilon(t)-y^\varepsilon(t)-z^\varepsilon(t)\}\\
  =&\delta D_x(t)\xi^\varepsilon(t)\chi_{E_\varepsilon}(t)+D_x(t)\zeta^\varepsilon(t).
\end{aligned}
\end{equation}
Moreover,
\begin{equation}\label{Second Frechet of D}
  \begin{aligned}
  &\frac{1}{2}\widetilde{D}_{xx}(t)(\xi^\varepsilon(t),\xi^\varepsilon(t))-\frac{1}{2}D_{xx}(t)(y^\varepsilon(t),y^\varepsilon(t))\\
  =&\frac{1}{2}\widetilde{D}_{xx}(t)(\xi^\varepsilon(t),\xi^\varepsilon(t))-\frac{1}{2}D_{xx}(t,\bar{x}(t), u^\varepsilon(t))(\xi^\varepsilon(t),\xi^\varepsilon(t))\\
 &+\frac{1}{2}D_{xx}(t,\bar{x}(t), u^\varepsilon(t))(\xi^\varepsilon(t),\xi^\varepsilon(t))-\frac{1}{2}D_{xx}(t,\bar{x}(t), \bar{u}(t))(\xi^\varepsilon(t),\xi^\varepsilon(t))\\
 &+\frac{1}{2}D_{xx}(t)(\xi^\varepsilon(t),\xi^\varepsilon(t))-\frac{1}{2}D_{xx}(t)(y^\varepsilon(t),y^\varepsilon(t))\\
  =&\frac{1}{2}\delta D_{xx}(t)(\xi^\varepsilon(t),\xi^\varepsilon(t))\chi_{E_\varepsilon}(t)+\frac{1}{2}D_{xx}(t)\left\{(\xi^\varepsilon(t),\xi^\varepsilon(t))-(y^\varepsilon(t),y^\varepsilon(t))\right\}\\
  &+\frac{1}{2}\left\{\widetilde{D}_{xx}(t)-D_{xx}(t,\bar{x}(t), u^\varepsilon(t))\right\}(\xi^\varepsilon(t),\xi^\varepsilon(t)).
\end{aligned}
\end{equation}
From \eqref{drift term of T-3-e}-\eqref{Second Frechet of D}, we conclude that
\begin{align*}
   & D(t,x^\varepsilon(t),u^\varepsilon(t))-D(t,\bar{x}(t),\bar{u}(t))-D_x(t)y^\varepsilon(t)-D_x(t)z^\varepsilon(t)\\
   &-\delta D(t)\chi_{E_\varepsilon}(t)-\frac{1}{2}D_{xx}(t)(y^\varepsilon(t),y^\varepsilon(t)) \\
  = &D_x(t)\zeta^\varepsilon(t)+\delta D_x(t)\xi^\varepsilon(t)\chi_{E_\varepsilon}(t)+\frac{1}{2}\left\{\widetilde{D}_{xx}(t)-D_{xx}(t,\bar{x}(t), u^\varepsilon(t))\right\}(\xi^\varepsilon(t),\xi^\varepsilon(t))\\
  &+\frac{1}{2}\delta D_{xx}(t)(\xi^\varepsilon(t),\xi^\varepsilon(t))\chi_{E_\varepsilon}(t)+\frac{1}{2}D_{xx}(t)\left\{(\xi^\varepsilon(t),\xi^\varepsilon(t))-(y^\varepsilon(t),y^\varepsilon(t))\right\}.
\end{align*}
Similarly, the diffusion terms are
\begin{align*}
   & F(t,x^\varepsilon(t),u^\varepsilon(t))-F(t,\bar{x}(t),\bar{u}(t))-F_x(t)y^\varepsilon(t)-F_x(t)z^\varepsilon(t)\\
&-\delta F(t)\chi_{E_\varepsilon}(t)-\delta F_x(t)y^\varepsilon(t)\chi_{E_\varepsilon}(t)-\frac{1}{2}F_{xx}(t)(y^\varepsilon(t),y^\varepsilon(t)) \\
  = &F_x(t)\zeta^\varepsilon(t)+\delta F_x(t)\eta^\varepsilon(t)\chi_{E_\varepsilon}(t)+\frac{1}{2}\left\{\widetilde{F}_{xx}(t)-F_{xx}(t,\bar{x}(t), u^\varepsilon(t))\right\}(\xi^\varepsilon(t),\xi^\varepsilon(t))\\
  &+\frac{1}{2}\delta F_{xx}(t)(\xi^\varepsilon(t),\xi^\varepsilon(t))\chi_{E_\varepsilon}(t)+\frac{1}{2}F_{xx}(t)\left\{(\xi^\varepsilon(t),\xi^\varepsilon(t))-(y^\varepsilon(t),y^\varepsilon(t))\right\},
\end{align*}
and
\begin{align*}
   & G(t,x^\varepsilon(t),u^\varepsilon(t))-G(t,\bar{x}(t),\bar{u}(t))-G_x(t)y^\varepsilon(t)-G_x(t)z^\varepsilon(t)\\
&-\delta G(t)\chi_{E_\varepsilon}(t)-\delta G_x(t)y^\varepsilon(t)\chi_{E_\varepsilon}(t)-\frac{1}{2}G_{xx}(t)(y^\varepsilon(t),y^\varepsilon(t)) \\
  = &G_x(t)\zeta^\varepsilon(t)+\delta G_x(t)\eta^\varepsilon(t)\chi_{E_\varepsilon}(t)+\frac{1}{2}\left\{\widetilde{G}_{xx}(t)-G_{xx}(t,\bar{x}(t), u^\varepsilon(t))\right\}(\xi^\varepsilon(t),\xi^\varepsilon(t))\\
  &+\frac{1}{2}\delta G_{xx}(t)(\xi^\varepsilon(t),\xi^\varepsilon(t))\chi_{E_\varepsilon}(t)+\frac{1}{2}G_{xx}(t)\left\{(\xi^\varepsilon(t),\xi^\varepsilon(t))-(y^\varepsilon(t),y^\varepsilon(t))\right\}.
\end{align*}
From \eqref{equation-T-3-e}, we have 
\begin{equation}\label{estimate-T-3-e}
  \sup_{t\in[0,T]}\|\zeta^\varepsilon(t)\|_{p}^2\leq \mathcal{C}\left(\left(\int_0^T\|\Theta_1(t)\|_{p}dt\right)^2+\int_0^T\left\{\|\Theta_2(t)\|_{p}^2+\|\Theta_3(t)\|_{p}^2\right\}dt\right).
\end{equation}

Next, we consider $\Theta_1(t), \Theta_2(t)$ and $\Theta_3(t)$, respectively. From \eqref{estimate-x-p-e}-\eqref{estimate-x-y-p-e}, we first estimate $\Theta_1(\cdot)$.
By the Minkowski inequality and the H\"{o}lder inequality again,
\begin{equation}\label{estimate-T-3-1-e}
\begin{aligned}
\int_0^T\|\Theta_1(t)\|_{p}dt
&\leq \mathcal{C}\int_0^T\left(\frac{1}{2}\left\|D_{xx}(t)\{(\xi^\varepsilon(t),\xi^\varepsilon(t))-(y^\varepsilon(t),y^\varepsilon(t))\}\right\|_{p}\right.\\
&\indent+\frac{1}{2}\left\|\left\{\widetilde{D}_{xx}(t)-D_{xx}(t,\bar{x}(t),u^\varepsilon(t))\right\}(\xi^\varepsilon(t),\xi^\varepsilon(t))\right\|_{p}\\
&\indent\left.+\chi_{E_\varepsilon}(t)\left\{\|\delta D_x(t)\xi^\varepsilon(t)\|_{p}+\frac{1}{2}\|\delta D_{xx}(t)(\xi^\varepsilon(t),\xi^\varepsilon(t))\|_{p}\right\}\right)dt.
\end{aligned}
\end{equation}
Now, we now estimate each term in the right of \eqref{estimate-T-3-1-e} separately. Based on \eqref{final estimate of xe-x even}, we have the following estimate:
\begin{equation}\label{estimate of drift term-1}
\begin{aligned}
& \int_0^T\chi_{E_\varepsilon}(t)\left\{\left\|\delta D_x(t)\xi^\varepsilon(t)\right\|_{p}+\frac{1}{2}\left\|\delta D_{xx}(t)(\xi^\varepsilon(t),\xi^\varepsilon(t))\right\|_{p}\right\}dt\\
\leq &\int_0^T\chi_{E_\varepsilon}(t)\left\{\|\delta D_x(t)\|_{\mathcal{L}(L^p(\mathscr{C}))}\|\xi^\varepsilon(t)\|_{p}+\|\delta D_{xx}(t)\|_{\mathcal{L}(L^p(\mathscr{C}),L^p(\mathscr{C});L^p(\mathscr{C}))}\|\xi^\varepsilon(t)\|_p^2\right\}dt\\
\leq &\mathcal{C}\int_0^t\chi_{E_\varepsilon}(t)\left\{\sup_{t\in[0,T]}\|\xi^\varepsilon(t)\|_{p}+\sup_{t\in[0,T]}\|\xi^\varepsilon(t)\|_{p}^2\right\}dt\\
\leq & \mathcal{C}\varepsilon^\frac{3}{2}.
 \end{aligned}
\end{equation}
From \eqref{abbreviation}, we can obtain that, for $t\in[0,T]$
\begin{equation}\label{estimate of drift term-2-operator}
\begin{aligned}
&\left\|\widetilde{D}_{xx}(t)-D_{xx}(t,\bar{x}(t),u^\varepsilon(t))\right\|_{\mathcal{L}(L^p(\mathscr{C}),L^p(\mathscr{C});L^p(\mathscr{C}))} \\
=&\left\|2\int_0^1(1-\theta)D_{xx}(t,\bar{x}(t)+\theta \xi^\varepsilon(t),u^\varepsilon(t))d\theta-D_{xx}(t,\bar{x}(t),u^\varepsilon(t))\right\|_{\mathcal{L}(L^p(\mathscr{C}),L^p(\mathscr{C});L^p(\mathscr{C}))}\\
= &\left\|2\int_0^1(1-\theta)\chi_{E_\varepsilon}(t)\{D_{xx}(t,\bar{x}(t)+\theta \xi^\varepsilon(t),u(t))-D_{xx}(t,\bar{x}(t),u(t))\}d\theta\right.\\
&\indent\left.+2\int_0^1(1-\theta)\left\{D_{xx}(t,\bar{x}(t)+\theta \xi^\varepsilon(t),\bar{u}(t))-D_{xx}(t)\right\}d\theta
\right\|_{\mathcal{L}(L^p(\mathscr{C}),L^p(\mathscr{C});L^p(\mathscr{C}))}\\
\leq&\mathcal{C}\left(\int_0^1\left\|D_{xx}(t,\bar{x}(t)+\theta \xi^\varepsilon(t),\bar{u}(t))-D_{xx}(t)\right\|_{\mathcal{L}(L^p(\mathscr{C}),L^p(\mathscr{C});L^p(\mathscr{C}))}d\theta+\chi_{E_\varepsilon}(t)\right).
 \end{aligned}
\end{equation}
Hence, by \eqref{estimate-x-p-e} and  the continuity of $D_{xx}(t,x,u)$ with respect to $x$, we have
\begin{equation}\label{estimate of drift term-2}
\begin{aligned}
&\left(\int_0^T\left\|\left\{\widetilde{D}_{xx}(t)-D_{xx}(t,\bar{x}(t),u^\varepsilon(t))\right\}(\xi^\varepsilon(t),\xi^\varepsilon(t))\right\|_pdt\right)^2\\
\leq&\mathcal{C}\int_0^T\left\|\widetilde{D}_{xx}(t)-D_{xx}(t,\bar{x}(t),u^\varepsilon(t))\right\|^2_{\mathcal{L}(L^p(\mathscr{C}),L^p(\mathscr{C});L^p(\mathscr{C}))}\|\xi^\varepsilon(t)\|_p^4dt\\
\leq&\mathcal{C}\sup_{t\in[0,T]}\|\xi^\varepsilon(t)\|_p^4\int_0^T\bigg\{\chi_{E_\varepsilon}(t)\\
&\left.+\int_0^1\|D_{xx}(t,\bar{x}(t)+\theta \xi^\varepsilon(t),\bar{u}(t))-D_{xx}(t)\|^2_{\mathcal{L}(L^p(\mathscr{C}),L^p(\mathscr{C});L^p(\mathscr{C}))}d\theta\right\}dt\\
\leq&\mathcal{C}\varepsilon^2.
\end{aligned}
\end{equation}
By means of \eqref{final estimate of xe-x even},\eqref{y-x0} and \eqref{finally-xe-xx-y}, and using  the H\"{o}lder inequality again, we obtain
\begin{equation}\label{estimate of drift term-3}
\begin{aligned}
&\left(\int_0^T\|D_{xx}(t)\{(\xi^\varepsilon(t),\xi^\varepsilon(t))-(y^\varepsilon(t),y^\varepsilon(t))\}\|_pdt \right)^2\\
=&\left(\int_0^T\|D_{xx}(t)(\xi^\varepsilon(t),\eta^\varepsilon(t))+D_{xx}(t)(\eta^\varepsilon(t),y^\varepsilon(t))\|_p dt\right)^2\\
\leq&\mathcal{C}\int_0^T\|D_{xx}(t)\|^2_{\mathcal{L}(L^p(\mathscr{C}),L^p(\mathscr{C});L^p(\mathscr{C}))}\{\|\xi^\varepsilon(t)\|_p^2+\|y^\varepsilon(t)\|_p^2\}\|\eta^\varepsilon(t)\|_p^2 dt\\
\leq&\mathcal{C}\left(\sup_{t\in[0,T]}\|\xi^\varepsilon(t)\|_p^2+\sup_{t\in[0,T]}\|y^\varepsilon(t)\|_p^2\right)\sup_{t\in[0,T]}\|\eta^\varepsilon(t)\|_p^2\\
\leq&\mathcal{C}\varepsilon^3.
\end{aligned}
\end{equation}
From \eqref{estimate-T-3-1-e}-\eqref{estimate of drift term-3}, we infer that
\begin{equation}\label{estimate-T-3-1'-even}
\left(\int_0^T\|\Theta_1(t)\|_{p}dt\right)^2\leq\mathcal{C}\varepsilon^2.
\end{equation}

By virtue of \eqref{final estimate of xe-x even} and \eqref{finally-xe-xx-y} again, similar to \eqref{estimate of drift term-1}, we obtain that
\begin{equation}\label{estimate of diffusion term of F-1-even}
\begin{aligned}
 &\int_0^T \chi_{E_\varepsilon}(t)\left\{\|\delta F_x(t)\eta^\varepsilon(t)\|_{p}^2+\|\delta F_{xx}(t)(\xi^\varepsilon(t),\xi^\varepsilon(t))\|_{p}^2\right\}dt\\
 \leq&\mathcal{C}\int_0^T\chi_{E_\varepsilon}(t)\left\{\sup_{t\in[0,T]}\|\eta^\varepsilon(t)\|_p^2+\sup_{t\in[0,T]}\|\xi^\varepsilon(t)\|_p^4\right\}dt\\
 \leq&\mathcal{C}\varepsilon^3.
 \end{aligned}
\end{equation}
Similar to \eqref{estimate of drift term-2}, we find that
\begin{equation}\label{estimate of diffusion term of F-2-even}
\begin{aligned}
 &\int_0^T \left\|\left(\widetilde{F}_{xx}(t)-F_{xx}(t,\bar{x}(t),u^\varepsilon(t))\right)(\xi^\varepsilon(t),\xi^\varepsilon(t))\right\|_{p}^2dt\\
\leq&\int_0^T\left\|\widetilde{F}_{xx}(t)-F_{xx}(t,\bar{x}(t),u^\varepsilon(t))\right\|_{\mathcal{L}(L^p(\mathscr{C}),L^p(\mathscr{C});L^p(\mathscr{C}))}^2\|\xi^\varepsilon(t)\|_p^4dt\\
 \leq&\mathcal{C}\varepsilon^2.
 \end{aligned}
\end{equation}
Similar to \eqref{estimate of drift term-3}, it holds that
\begin{equation}\label{estimate of diffusion term of F-3-even}
\begin{aligned}
 &\int_0^T\|F_{xx}(t)\{(\xi^\varepsilon(t),\xi^\varepsilon(t))-(y^\varepsilon(t),y^\varepsilon(t))\}\|_{p}^2dt\\
 \leq&\int_0^T\|F_{xx}(t)\|^2_{\mathcal{L}(L^p(\mathscr{C}),L^p(\mathscr{C});L^p(\mathscr{C}))}\left\{\|\xi^\varepsilon(t)\|_p^2+\|y^\varepsilon(t)\|_p^2\right\}\|\eta^\varepsilon(t)\|_p^2 dt\\
 \leq &\mathcal{C}\varepsilon^3.
 \end{aligned}
\end{equation}
Thus, from \eqref{estimate of diffusion term of F-1-even}-\eqref{estimate of diffusion term of F-3-even},
we conclude that
\begin{equation}\label{estimate-T-3-2'-even}
\begin{aligned}
\int_0^T\|\Theta_2(t)\|_{p}^2dt  
&\leq\mathcal{C} \int_0^T\Big(\|F_{xx}(t)\{(\xi^\varepsilon(t),\xi^\varepsilon(t))-(y^\varepsilon(t),y^\varepsilon(t))\}\|_{p}^2\\
&\indent\indent +\left\|\left\{\widetilde{F}_{xx}(t)-F_{xx}(t,\bar{x}(t),u^\varepsilon(t))\right\}(\xi^\varepsilon(t),\xi^\varepsilon(t))\right\|_{p}^2\\
&\indent +\chi_{E_\varepsilon}(t)\left\{\|\delta F_x(t)\eta^\varepsilon(t)\|_{p}^2+\|\delta F_{xx}(t)(\xi^\varepsilon(t),\xi^\varepsilon(t))\|_{p}^2\right\}\Big)dt\\
&\leq \mathcal{C} \varepsilon^2.
\end{aligned}
\end{equation}
By a similar arguments, we have
\begin{equation}\label{estimate-T-3-3'-even}
\int_0^T\|\Theta_3(t)\|_{p}^2dt\leq\mathcal{C} \varepsilon^2.
\end{equation}
Substituting \eqref{estimate-T-3-1'-even}, \eqref{estimate-T-3-2'-even}, \eqref{estimate-T-3-3'-even} into \eqref{estimate-T-3-e}, we obtain that
\begin{equation*}
 \sup_{t\in[0,T]}\|\zeta^\varepsilon(t)\|_p^2\leq\mathcal{C}\varepsilon^2.
\end{equation*}
Then \eqref{estimate-x-y-z-p-e} holds.

Finally, we prove \eqref{cost functional-J-even}. By Taylor expansion, we have
\begin{align*}
& \mathcal{J}(u^\varepsilon(\cdot))-\mathcal{J}(\bar{u}(\cdot))  \\
&=h(x^\varepsilon(T))-h(\bar{x}(T))+\int_0^T\left\{L(t,x^\varepsilon(t),u^\varepsilon(t))-L(t,\bar{x}(t),\bar{u}(t))\right\}dt\\
&=\textrm{Re}\int_0^T\left\{\delta L(t)\chi_{E_\varepsilon}(t)+\left\langle L_{x}(t,\bar{x}(t),u^\varepsilon(t)),\xi^\varepsilon(t)\right\rangle\right\} dt\\
&\indent+\textrm{Re}\int_0^T\left\{\int_0^1\left\langle(1-\theta)L_{xx}(t,\bar{x}(t)+\theta\xi^\varepsilon(t),u^\varepsilon(t))\xi^\varepsilon(t),\xi^\varepsilon(t)\right\rangle d\theta \right \}dt\\
&\indent+\textrm{Re}\left\langle h_x(\bar{x}(T)),\xi^\varepsilon(T)\right\rangle+\textrm{Re}\int_0^1\left\langle(1-\theta) h_{xx}(\bar{x}(T)+\theta\xi^\varepsilon(T))\xi^\varepsilon(T),\xi^\varepsilon(T)\right\rangle d\theta.
\end{align*}
This, together with the definitions of $\xi^\varepsilon, y^\varepsilon, z^\varepsilon, \eta^\varepsilon$ and $\zeta^\varepsilon$, yields that
\begin{align*}
& \mathcal{J}(u^\varepsilon(\cdot))-\mathcal{J}(\bar{u}(\cdot))  \\
=&\textrm{Re}\int_0^T\Big\{\delta L(t)\chi_{E_\varepsilon}(t)+\left\langle \delta L_{x}(t),\xi^\varepsilon(t)\right\rangle\chi_{E_\varepsilon}(t)+\left\langle L_{x}(t), y^\varepsilon(t)+z^\varepsilon(t)\right\rangle+\left\langle L_{x}(t),\zeta^\varepsilon(t)\right\rangle\\
&+\int_0^1\left\langle(1-\theta)\left\{L_{xx}(t,\bar{x}(t)+\theta\xi^\varepsilon(t), u^\varepsilon(t))-L_{xx}(t, \bar{x}(t),u^\varepsilon(t))\right\}\xi^\varepsilon(t),\xi^\varepsilon(t)\right\rangle d\theta\\
&\indent+\frac{1}{2}\left\langle\delta L_{xx}(t)\xi^\varepsilon(t),\xi^\varepsilon(t)\right\rangle\chi_{E_\varepsilon}(t)+\frac{1}{2}\left\langle L_{xx}(t)\eta^\varepsilon(t),y^\varepsilon(t)\right\rangle+\frac{1}{2}\left\langle L_{xx}(t)y^\varepsilon(t),y^\varepsilon(t)\right\rangle\\
&\indent+\frac{1}{2}\left\langle L_{xx}(t)\xi^\varepsilon(t),\eta^\varepsilon(t)\right\rangle\Big\}dt+\textrm{Re}\left\langle h_x(\bar{x}(T)),y^\varepsilon(T)+z^\varepsilon(T)\right\rangle\\
&+\textrm{Re}\left\langle h_x(\bar{x}(T)),\zeta^\varepsilon(T)\right\rangle+\frac{1}{2}\textrm{Re}\left\langle h_{xx}(\bar{x}(T))\xi^\varepsilon(T),\eta^\varepsilon(T) \right\rangle\\
&+\frac{1}{2}\textrm{Re}\left\langle h_{xx}(\bar{x}(T))\eta^\varepsilon(T),y^\varepsilon(T) \right\rangle+\frac{1}{2}\textrm{Re}\left\langle h_{xx}(\bar{x}(T))y^\varepsilon(T),y^\varepsilon(T) \right\rangle\\
&+\textrm{Re}\int_0^1\left\langle(1-\theta)\left\{h_{xx}(\bar{x}(T)+\theta\xi^\varepsilon(T))-h_{xx}(\bar{x}(T))\right\}\xi^\varepsilon(T),\xi^\varepsilon(T)\right\rangle d\theta.
\end{align*}
Similar to \eqref{estimate of drift term-2-operator}, for $t\in[0,T]$, we find that
\begin{equation}\label{second order of Taylor expansion of L}
\begin{aligned}
 & \left\|\int_0^1(1-\theta)\left\{L_{xx}(t,\bar{x}(t)+\theta\xi^\varepsilon(t), u^\varepsilon(t))-L_{xx}(t, \bar{x}(t),u^\varepsilon(t))\right\}d\theta\right\|_{\mathcal{L}(L^p(\mathscr{C});L^{p'}(\mathscr{C}))}\\
 &\leq \mathcal{C}\int_0^1\left\|{L_{xx}(t,\bar{x}(t)+\theta\xi^\varepsilon(t), u^\varepsilon(t))-L_{xx}(t)}\right\|_{L(L^p(\mathscr{C});L^{p'}(\mathscr{C}))} d\theta+\chi_{E_\varepsilon}(t).
\end{aligned}
\end{equation}
From \eqref{estimate-x-p-e}-\eqref{estimate-x-y-z-p-e} and \eqref{second order of Taylor expansion of L}, we deduce that
\begin{equation*}
\begin{aligned}
 & \mathcal{ J}(u^\varepsilon(\cdot))-\mathcal{J}(\bar{u}(\cdot))  \\
   =&\textrm{Re}\int_0^T\left\{\delta L(t)\chi_{E_\varepsilon}(t)+\left\langle L_{x}(t), y^\varepsilon(t)+z^\varepsilon(t)\right\rangle+\frac{1}{2}\left\langle L_{xx}(t)y^\varepsilon(t),y^\varepsilon(t)\right\rangle\right\}dt\\
 &+\textrm{Re}\left\langle h_x(\bar{x}(T)),y^\varepsilon(T)+z^\varepsilon(T)\right\rangle+\frac{1}{2}\textrm{Re}\left\langle h_{xx}(\bar{x}(T))y^\varepsilon(T),y^\varepsilon(T) \right\rangle+\textit{\textbf{o}}(\varepsilon).\\
\end{aligned}
\end{equation*}
\end{proof}

In order to establish the necessary conditions for an optimal pair $(\bar{x}(\cdot),\bar{u}(\cdot))$ of $\textbf{Problem (QOC)}$, we need to introduce the following BQSDEs:
\begin{equation}\label{BSDE-p-even}
\left\{
\begin{aligned}
d\phi(t)=&-\left\{D_{x}(t)^*\phi(t)+\left(F^e_x(t)-F^o_x(t)+G_x(t)\right)^*\left(\Phi(t)_e-\Phi(t)_o\right)-L_x(t)\right\}dt\\
&+\Phi(t) dW(t),  \ \rm{in}\ [0,T),\\
\phi(T)=&-h_{x}(\bar{x}(T)),
\end{aligned}
\right.
\end{equation}
and
\begin{equation}\label{BSDE-P-even}
\left\{
\begin{aligned}
dP(t)=&-\left\{D_{x}(t)^*P(t)+P(t) D_x(t)+(F^e_x(t)-F^o_x(t)+G_x(t))^*Q(t)\right.\\
&\indent+\left(F^e_x(t)-F^o_x(t)+G_x(t)\right)^*P(t)\left(F_x(t)+G^e_x(t)-G^o_x(t)\right)\\
&\indent\left.+Q(t)(F^e_x(t)-F^o_x(t)+G_x(t))+\mathbb{H}_{xx}(t,\bar{x}(t),\bar{u}(t),\phi(t),\Phi(t))\right\}dt\\
&+Q(t) dW(t),\ \rm{in}\ [0,T),\\
P(T)=&-h_{xx}(\bar{x}(T)),
\end{aligned}
\right.
\end{equation}
where 
the map $\Phi(\cdot):[0,T]\to \mathcal{H}^{p'}(0,T)$ is adapted,
and $F_x^e(\cdot)-F_x^o(\cdot):=\Upsilon \circ F_x(\cdot)$, $G_x^e(\cdot)-G_x^o(\cdot):=\Upsilon \circ G_x(\cdot)$, the map $Q(\cdot):[0,T]\to\mathcal{L}(L^p(\mathscr{C});L^{p'}(\mathscr{C}))$.
And the Hamiltonian function $\mathbb{H}(\cdot,\cdot,\cdot,\cdot,\cdot)$ is defined by
\begin{equation}\label{Hamiltonian function of even}
\begin{aligned}
&\mathbb{H}(t,x,u,\phi,\Phi)\\
:&=\langle \phi,D(t,x,u)\rangle+\langle \Phi_e-\Phi_{o}, F(t,x,u)_e-F(t,x,u)_o+G(t,x,u)\rangle-L(t,x,u),\\
&\indent\indent (t,x,u,\phi,\Phi)\in [0,T]\times L^p(\mathscr{C})\times U \times L^{p'}(\mathscr{C})\times L^{p'}(\mathscr{C}).
\end{aligned}
\end{equation}
Hereafter we use $\langle\cdot,\cdot\rangle$ to represent \textit{the dual product} of $L^{p'}(\mathscr{C})$ and $L^p(\mathscr{C})$, and $\langle x,y\rangle=m(x^*y)$ for $x\in L^{p'}(\mathscr{C}), y\in L^p(\mathscr{C})$.
\begin{thm}\label{Stochastic Maximum Principle}
Suppose that Assumption \ref{Assump 1} holds. Let $(\bar{x}(\cdot),\bar{u}(\cdot))$ be an optimal pair of \textbf{Problem (QOC)}. Let $(\phi(\cdot), \Phi(\cdot))$ be the solution to  \eqref{BSDE-p-even}, and $(P(\cdot), Q(\cdot))$
be the solution to  \eqref{BSDE-P-even}. Then,
\begin{equation}\label{SMP-inequality-even}
  \begin{aligned}
 & {\rm{Re}}\mathbb{H}(t, \bar{x}(t),\bar{u}(t),\phi(t), \Phi(t))-{\rm{Re}}\mathbb{H}(t, \bar{x}(t),u, \phi(t), \Phi(t))\\
  &-\frac{1}{2}{\rm{Re}}\left\langle (P^e(t)-P^o(t))\{(F(t,\bar{x}(t),\bar{u}(t))+G(t,\bar{x}(t),\bar{u}(t))_e-G(t,\bar{x}(t),\bar{u}(t)))_o\right.\\
  &\indent\indent-(F(t,\bar{x}(t),u)+G(t,\bar{x}(t),u)_e-G(t,\bar{x}(t),u)_o)\}, \\
 &\indent\indent\{(F(t,\bar{x}(t),\bar{u}(t))_e-F(t,\bar{x}(t),\bar{u}(t))_o+G(t,\bar{x}(t),\bar{u}(t)))\\
 &\indent\indent\left.-( F(t,\bar{x}(t),u)_e-F(t,\bar{x}(t),u)_o+G(t,\bar{x}(t),u))\}\right\rangle\geq 0, \ {\rm a.e.}\ t\in[0,T],\  u \in U. 
\end{aligned}
\end{equation}
\end{thm}
\begin{proof}
In view of the dual relation between \eqref{FSDE-y-e} and \eqref{BSDE-p-even}, we have
\begin{equation}\label{Ito formula-p-y-even}
\begin{aligned}
&\langle \phi(T), y^\varepsilon(T)\rangle= - \left\langle h_x(\bar{x}(T)), y^\varepsilon(T)\right\rangle\\
=&-\int_0^T\Big\{\left\langle D_x(t)^*\phi(t), y^\varepsilon(t)\right\rangle+\left\langle \Phi(t)_e-\Phi(t)_o, (F_x^e(t)-F_x^o(t)+G_x(t))y^\varepsilon(t)\right\rangle\Big\}dt\\
&+\int_0^T\Big\{\langle\Phi(t)_e-\Phi(t)_o, \delta(F(t)_e-F(t)_o+G(t))\rangle\chi_{E_\varepsilon}(t)+\langle\phi(t), D_x(t)y^\varepsilon(t) \rangle\Big\}dt\\
 &+\int_0^T\left\{\langle\Phi(t)_e-\Phi(t)_o, (F_x^e(t)-F_x^o(t)+G_x(t))y^\varepsilon(t)\rangle+\left\langle L_x(t), y^\varepsilon(t)\right\rangle \right\}dt\\
=&\int_0^T\Big\{\left\langle L_x(t), y^\varepsilon(t)\right\rangle+\langle \Phi(t)_e-\Phi(t)_o, \delta(F(t)_e-F(t)_o+G(t))\rangle\chi_{E_\varepsilon}(t)\Big\}dt,
\end{aligned}
\end{equation}
Similarly,
\begin{equation}\label{Ito formula-p-z-even}
\begin{aligned}
& \langle \phi(T), z^\varepsilon(T)\rangle= - \left\langle h_x(\bar{x}(T)),z^\varepsilon(T)\right\rangle\\
=&\int_0^T\Big\{\langle \phi(t), \delta D(t)\rangle+\langle \Phi(t)_e-\Phi(t)_o, \delta(F_x^e(t)-F_x^o(t)+G_x(t))y^\varepsilon(t)\rangle\Big\}\chi_{E_\varepsilon}(t) dt\\
& +\int_0^T  \left\langle L_x(t), z^\varepsilon(t)\right\rangle+\frac{1}{2}\Big\{\langle \phi(t),D_{xx}(t)(y^\varepsilon(t),y^\varepsilon(t))\rangle\\
& \indent\indent\indent\indent+\langle \Phi(t)_e-\Phi(t)_o, (F^e_{xx}(t)-F^o_{xx}(t)+G_{xx}(t))(y^\varepsilon(t),y^\varepsilon(t))\rangle\Big\} dt,
\end{aligned}
\end{equation}
where $F^e_{xx}(t)-F^o_{xx}(t)=\Upsilon\circ F_{xx}(t)$.
From \eqref{estimate-y-p-e}, \eqref{Ito formula-p-y-even} and \eqref{Ito formula-p-z-even}, we obtain that
\begin{equation}\label{Ito formula-p-y+z}
\begin{aligned}
     &\left\langle \phi(T), y^\varepsilon(T)+z^\varepsilon(T)\right\rangle= - \left\langle h_x(\bar{x}(T)), y^\varepsilon(T)+z^\varepsilon(T)\right\rangle\\
=&\int_0^T\Big\{
\langle \phi(t), \delta D(t)\rangle
+\langle \Phi(t)_e-\Phi(t)_o, \delta(F(t)_e-F(t)_o+G(t))\rangle\Big\}\chi_{E_\varepsilon}(t)dt\\
&+\int_0^T \frac{1}{2}\Big\{\langle \Phi(t)_e-\Phi(t)_o, (F^e_{xx}(t)-F^o_{xx}(t)+G_{xx}(t))(y^\varepsilon(t),y^\varepsilon(t))\rangle \\
&\indent\indent+\langle \phi(t),D_{xx}(t)(y^\varepsilon(t),y^\varepsilon(t)) \rangle\Big\}+\left\langle L_x(t), y^\varepsilon(t)+z^\varepsilon(t)\right\rangle dt+\textit{\textbf{o}}(\varepsilon^{\frac{3}{2}}).
\end{aligned}
\end{equation}
By \cite[Chapter 12.4]{L.Z-2020},  from \eqref{FSDE-y-e} and \eqref{BSDE-P-even}, we have the following dual product
\begin{equation}\label{the dual product of Py and y}
\begin{aligned}
 &\langle P(T)y^\varepsilon(T),y^\varepsilon(T) \rangle =-\left\langle h_{xx}(\bar{x}(T))y^\varepsilon(T), y^\varepsilon(T)\right\rangle \\
 =&\int_0^T\langle (P^e(t)-P^o(t))\delta(F(t)+G(t)_e-G(t)_o), \delta(F(t)_e-F(t)_o+G_x(t))\rangle\chi_{E_\varepsilon}(t)dt\\
&-2\int_0^T\langle Q^e(t)y^\varepsilon(t)_o+Q^o(t)y^\varepsilon(t)_e, (F_x^e(t)-F_x^o(t)+G_x(t))y^\varepsilon(t)\rangle dt\\
&-2\int_0^T\langle P^o(t)(F_x(t)+G_x^e(t)-G^o_x(t))y^\varepsilon(t), (F_x^e(t)-F_x^o(t)+G_x(t))y^\varepsilon(t)\rangle dt\\
&-\int_0^T\langle\mathbb{H}_{xx}(t,\bar{x}(t),\bar{u}(t),\phi(t),\Phi(t)) y^\varepsilon(t),\ y^\varepsilon(t)\rangle dt+\textit{\textbf{o}}(\varepsilon^{\frac{3}{2}}).
\end{aligned}
\end{equation}
Substituting \eqref{the dual product of Py and y} into \eqref{cost functional-J-even}, and combining with \eqref{estimate-y-p-e} and \eqref{estimate-z-p-e}, we have
\begin{equation}\label{Finally-estiamte-even}
  \begin{aligned}
    0\leq& \mathcal{J}(u^\varepsilon(\cdot))-\mathcal{J}(\bar{u}(\cdot))  \\
     =&-\textrm{Re}\int_0^T\chi_{E_\varepsilon}(t) \Big\{-\delta L(t)\langle \phi(t), \delta D(t)\rangle+\langle \Phi(t)_e-\Phi(t)_o, \delta(F(t)_e-F(t)_o+G(t))\rangle\\
     &-\frac{1}{2}\langle (P^e(t)-P^o(t))\delta(F(t)+G(t)_e-G(t)_o), \delta(F(t)_e-F(t)_o+G(t))\rangle\Big\}dt+\textit{\textbf{o}}(\varepsilon).
   \end{aligned}
\end{equation}
Therefore, we easily obtain
\begin{equation*}
\begin{aligned}
  &\textrm{Re}\int_0^T\Bigg\{\frac{1}{2}\left\langle(P^e(t)-P^o(t))\delta(F(t)+G(t)_e-G(t)_o), \delta(F(t)_e-F(t)_o+G(t))\right\rangle+\delta \mathbb{H}(t)\Bigg\}\chi_{E_\varepsilon}(t)dt\\
  &\leq\textit{\textbf{o}}(\varepsilon),
  \end{aligned}
\end{equation*}
where $$\delta \mathbb{H}(t):=\mathbb{H}(t, \bar{x}(t),u(t), \phi(t), \Phi(t))-\mathbb{H}(t, \bar{x}(t),\bar{u}(t),\phi(t), \Phi(t)).$$
Let $\varepsilon\to 0$, for a.e. $t\in[0,T]$, we have
\begin{equation*}
\begin{aligned}
 &\textrm{Re} \mathbb{H}(t, \bar{x}(t),u(t),\phi(t), \Phi(t))-\textrm{Re}\mathbb{H}(t, \bar{x}(t),\bar{u}(t), \phi(t), \Phi(t))\\
  &+\frac{1}{2}\textrm{Re}\left\langle(P^e(t)-P^o(t))\delta(F(t)+G(t)_e-G(t)_o), \delta(F(t)_e-F(t)_o+G(t))\right\rangle\leq 0.\\
\end{aligned}
\end{equation*}
Thus \eqref{SMP-inequality-even} is proved.
\end{proof}
\begin{rem}
In \eqref{SMP-inequality-even}, 
\begin{equation}\label{noncommute}
\langle(P^e(t)-P^o(t))\delta(F(t)+G(t)_e-G(t)_o), \delta(F(t)_e-F(t)_o+G(t))\rangle
\end{equation}
indicates that the operators do not commute with each other.
If the maps $F(\cdot,\cdot,\cdot)$, $G(\cdot,\cdot,\cdot):[0,T]\times L^p(\mathscr{C})\times U\to L^p(\mathscr{C}_e)$ are adapted, then $F(\cdot,\cdot,\cdot),G(\cdot,\cdot,\cdot)$ can commute  with $dW(t)$, and \eqref{noncommute} can be transformed into
\begin{equation*}
\langle(P^e(t)-P^o(t))\delta(F(t)+G(t)), \delta(F(t)+G(t))\rangle.
\end{equation*}
Then, Theorem \ref{Stochastic Maximum Principle} is similar to the classical results \cite[Theorem 2.6]{D.M} and \cite[Theorem 12.17]{L.Z-2020}.
\end{rem}
\section{The Solutions to Backward Quantum Stochastic Differential Equations}\label{solution-QBSDE}
In section \ref{Principle}, we present a necessary condition for quantum optimal control of \eqref{FSDE}. That is,
the solutions to BQSDEs is necessary for the Pontryagin-type Maximum Principle. In this section, we investigate the solution to \eqref{BSDE-p-even} in $L^{p'}(\mathscr{C})$ for $p'\in(1,2]$.
Let us introduce the following semilinear BQSDE
\begin{equation}\label{BQSDE}
\left\{
\begin{aligned}
  dy(t)&=f(t,y(t),Y(t))dt+Y(t)dW(t),\ {\rm in}\ [0,T),\\
 y(T)&=y_T,
\end{aligned}
\right.
\end{equation}
where $y_T\in L^{p'}(\mathscr{C}_T)$ is given, the map $f(\cdot,\cdot,\cdot):[0,T]\times L^{p'}(\mathscr{C})\times L^{p'}(\mathscr{C})\to  
L^{p'}(\mathscr{C})$ is adapted, $Y(\cdot)\in\mathcal{H}^{p'}(0,T)$.
To obtain the solution to \eqref{BQSDE}, we assume the following condition:
\begin{Assump}
 \label{Assump 3}
    The map $f(\cdot, y(\cdot), Y(\cdot))$ is adapted for each $y(t)\in L^{p'}(\mathscr{C}_t)$ and $Y(t)\in L^{p'}(\mathscr{C}_t)$, 
$f(\cdot,0,0)\in L^{p'}_\mathbb{A}(\mathscr{C};L^1(0,T))$, and there exist nonnegative functions $g_1(\cdot)\in L^1(0,T)$ and $g_2(\cdot)\in L^2(0,T)$ such that
\begin{align*}
  \|f(t,y_1,Y_1)-f(t,y_2,Y_2)\|_{p'}&\leq g_1(t)\|y_1-y_2\|_{p'}+g_2(t)\|Y_1-Y_2\|_{p'} ,\\
  \ \forall y_1,\ y_2\in L^{p'}(\mathscr{C}_t),&\  Y_1,\ Y_2\in  L^{p'}(\mathscr{C}_t).
\end{align*}
\end{Assump}
\begin{defn}
A pair $(y(\cdot),Y(\cdot))\in C_{\mathbb{A}}(0,T; L^{p'}(\mathscr{C}))\times \mathcal{H}^{p'}(0,T)$ is called a solution to \eqref{BQSDE} if $y(\cdot)$ is an  $L^{p'}(\mathscr{C})$-valued adapted, continuous process, $Y(\cdot)\in \mathcal{H}^{p'}(0,T)$, $f(\cdot, y(\cdot), Y(\cdot))
\in L^{p'}_\mathbb{A}(\mathscr{C};L^1(0,T))$ {\rm a.s.}, and for any $t\in[0, T]$,
\begin{equation}\label{def of QBSDE}
  y(t)=y_T-\int_t^Tf(\tau,y(\tau),Y(\tau))d\tau-\int_t^TY(\tau)dW(\tau).
\end{equation}
\end{defn}
\begin{thm}\label{the solution to QBSDEs}
Suppose that Assumption \ref{Assump 3} holds. Then, there exists a unique pair $(y(\cdot), Y(\cdot))\in C_\mathbb{A}(0,T; L^{p'}(\mathscr{C}))\times \mathcal{H}^{p'}(0,T)$ which satisfies \eqref{def of QBSDE}. Furthermore, 
\begin{equation}\label{Estimate-back}
 \|(y(\cdot),Y(\cdot))\|_{C_{\mathbb{A}}(0,T; L^{p'}(\mathscr{C}))\times \mathcal{H}^{p'}(0,T)}\leq \mathcal{C}\left(\|y_T\|_{p'}+\int_0^T\|f(\tau,0,0)\|_{p'}d\tau\right).
\end{equation}
\end{thm}
\begin{proof}
We divide the proof into two steps.

\textbf{Step 1.} First, we claim that for any $f(\cdot)\in L_\mathbb{A}^{p'}(\mathscr{C}; L^1(0,T))$ and $y_T\in L^{p'}(\mathscr{C})$, there is a pair $(y(\cdot),Y(\cdot))\in C_{\mathbb{A}}(0,T; L^{p'}(\mathscr{C}))\times \mathcal{H}^{p'}(0,T)$ such that
\begin{equation*}
y(t)=y_T-\int_t^Tf(\tau)d\tau-\int_t^TY(\tau)dW(\tau),\quad t\in[0,T).
\end{equation*}
Let
\begin{equation}\label{Martingale-M-x}
M(t):=\mathbb{E}\left(y_T-\int_0^Tf(\tau)d\tau\bigg|\mathscr{C}_t\right),\quad y(t):=\mathbb{E}\left(y_T-\int_t^Tf(\tau)d\tau\bigg|\mathscr{C}_t\right).
\end{equation}
Thus, $M(0)=y(0)=\mathbb{E}(M(T))$. By the noncommutative martingales representation theorem \cite[Theorem 4.6]{P.X}, 
there exists $Y(\cdot)\in \mathcal{H}^{p'}(0,T)$ such that
\begin{equation}\label{martingale representation of M}
M(t)=M(0)+\int_0^tY(\tau)dW(\tau),\quad  t\in(0, T].
\end{equation}
Hence,
\begin{equation}\label{estimate-X}
\left\|\int_0^TY(\tau)dW(\tau)\right\|_{p'}^2\leq\|M(T)\|_{p'}^2\leq2\left(\|y_T\|_{p'}^2+\left\|\int_0^Tf(\tau)d\tau\right\|_{p'}^2\right).
\end{equation}
From \eqref{martingale representation of M}, we have
\begin{align*}
  y_T-\int_0^Tf(\tau)d\tau&=M(0)+\int_0^TY(\tau)dW(\tau) \\
   &=y(0)+\int_0^TY(\tau)dW(\tau),
\end{align*}
and
\begin{align*}
  y(t)&=M(t)+\int_0^tf(\tau)d\tau\\
  &=y(0)+\int_0^tY(\tau)dW(\tau)+\int_0^tf(\tau)d\tau \\
  &=y_T-\int_t^Tf(\tau)d\tau-\int_t^TY(\tau)dW(\tau).
\end{align*}
Then, $y(\cdot)=y(0)+\int_0^\cdot f(\tau)d\tau+\int_0^\cdot Y(\tau)dW(\tau).$
It is clear that $y(\cdot)\in C_\mathbb{A}(0,T;L^{p'}(\mathscr{C}))$.
From \eqref{Martingale-M-x}, we obtain that
\begin{equation}\label{estimate-x}
\|y(\cdot)\|^2_{C_\mathbb{A}(0,T; L^{p'}(\mathscr{C}))}\leq  \mathcal{C}\left(\|y_T\|_{p'}^2+\left\|\int_0^Tf(t)dt\right\|_{p'}^2\right).
\end{equation}
On the other hand, from \eqref{estimate-X} and \eqref{estimate-x}, we have
\begin{equation}\label{estimate-x-X}
\|(y(\cdot),Y(\cdot)\|^2_{C_\mathbb{A}(0,T; L^{p'}(\mathscr{C}))\times \mathcal{H}^{p'}(0,T)}\leq \mathcal{C}\left(\|y_T\|^2_{p'}+\left\|\int_0^Tf(t)dt\right\|_{p'}^2\right).
\end{equation}

\textbf{Step 2.} We  prove that the pair $(y(\cdot),Y(\cdot))\in C_\mathbb{A}(0,T; L^{p'}(\mathscr{C}))\times \mathcal{H}^{p'}(0,T)$ is the unique solution to \eqref{BQSDE}. 
For fixed $T_1\in[0,T)$ and any $(z(\cdot),Z(\cdot))\in C_\mathbb{A}(T_1,T;L^{p'}(\mathscr{C}))$ $\times \mathcal{H}^{p'}(T_1,T)$. 
We consider the following BQSDE
\begin{equation}\label{NBQSDE}
\left\{
\begin{aligned}
  dy(t)&=f(t,z(t),Z(t))dt+Y(t)dW(t),\ {\rm in}\ [T_1,T),\\
 y(T)&=y_T.
\end{aligned}
\right.
\end{equation}
By the result in \textbf{Step 1}, the pair $(y(\cdot),Y(\cdot))\in C_\mathbb{A}(T_1,T;L^{p'}(\mathscr{C}))\times\mathcal{H}^{p'}(T_1,T)$ is the unique solution to \eqref{NBQSDE}. 
Then, we define a map
$$\Gamma: C_\mathbb{A}(T_1,T;L^{p'}(\mathscr{C}))\times \mathcal{H}^{p'}(T_1,T) \to C_\mathbb{A}(T_1,T;L^{p'}(\mathscr{C}))\times \mathcal{H}^{p'}(T_1,T)$$
by
\begin{equation*}
\Gamma(z(\cdot),Z(\cdot))= (y(\cdot), Y(\cdot)).
\end{equation*}
Next, we claim that, 
for $T_1$ being sufficiently close to $T$, 
\begin{equation}
\label{The fixed point}
\begin{aligned}
     &\|\Gamma(z,Z)-\Gamma(\bar{z},\bar{Z})\|_{C_\mathbb{A}(T_1,T;L^{p'}(\mathscr{C}))\times \mathcal{H}^{p'}(T_1,T)}\\
     &\indent\indent\leq \frac{1}{2}\|(z(\cdot)-\bar{z}(\cdot),Z(\cdot)-\bar{Z}(\cdot))\|_{C_\mathbb{A}(T_1,T;L^{p'}(\mathscr{C}))\times \mathcal{H}^{p'}(T_1,T)},\\
     &  \indent\indent\indent\forall (z,Z), (\bar{z},\bar{Z})\in C_\mathbb{A}(T_1,T; L^{p'}(\mathscr{C}))\times \mathcal{H}^{p'}(T_1,T).
\end{aligned}
\end{equation}
To show \eqref{The fixed point},  let
$$(\hat{y}(\cdot),\hat{Y}(\cdot)):=\Gamma(z,Z)-\Gamma(\bar{z},\bar{Z}),\quad
\hat{f}(\cdot):=f(\cdot,z(\cdot),Z(\cdot))-f(\cdot,\bar{z}(\cdot),\bar{Z}(\cdot)).$$
Then, $(\hat{y}(\cdot), \hat{Y}(\cdot))$ is solution to
\begin{equation}\label{equation-QSDE of [T_1,T]}
\left\{
\begin{aligned}
  d\hat{y}(t)&=\hat{f}(t)dt+\hat{Y}(t)dW(t),\ {\rm in}\ [T_1,T),\\
 \hat{y}(T)&=0.
\end{aligned}
\right.
\end{equation}
By  Assumption \ref{Assump 3}, the H\"{o}lder inequality and \eqref{Minkowski of p in (1,2]}, we have
\begin{equation}\label{estimate of contraction map}
  \begin{aligned}
    &\|(\hat{y}(\cdot), \hat{Y}(\cdot))\|^2_{C_\mathbb{A}(T_1,T; L^{p'}(\mathscr{C}))\times \mathcal{H}^{p'}(T_1,T)} \leq \mathcal{C}\left\|\int_{T_1}^T\hat{f}(t)dt\right\|^2_{p'} \\
      &\leq \mathcal{C}\left(\int_{T_1}^T\left\|f(t,z(t),Z(t))-f(t,\bar{z}(t),\bar{Z}(t))\right\|_{p'}dt\right)^2\\
      &\leq \mathcal{C}\|z(\cdot)-\bar{z}(\cdot)\|^2_{C_\mathbb{A}(T_1,T;L^{p'}(\mathscr{C}))}\left(\int _{T_1}^T|g_1(t)|dt\right)^2\\
      &\indent\indent\indent\indent+\mathcal{C}\|Z(\cdot)-\bar{Z}(\cdot)\|_{\mathcal{H}^{p'}(T_1,T)}^2\int_{T_1}^T|g_2(t)|^2dt\\
      &\leq \mathcal{C}\left\{\left(\int _{T_1}^T|g_1(t)|dt\right)^2+\int_{T_1}^T|g_2(t)|^2dt\right\}\\
      &\indent\indent\indent\indent\cdot\|(z(\cdot)-\bar{z}(\cdot),Z(\cdot)-\bar{Z}(\cdot))\|^2_{C_\mathbb{A}(T_1,T;L^{p'}(\mathscr{C}))\times \mathcal{H}^{p'}(T_1,T)},
  \end{aligned}
\end{equation}
and
\begin{equation}\label{x X estimate}
  \begin{aligned}
    &\|(y(\cdot), Y(\cdot))\|^2_{C_\mathbb{A}(T_1,T;L^{p'}(\mathscr{C}))\times\mathcal{H}^{p'}(T_1,T)}  \leq \mathcal{C}\left\{\|y_T\|^2_{p'}+\left(\int_{T_1}^T\|f(t,z(t),Z(t))\|_{p'} dt\right)^2\right\} \\
      &\leq \mathcal{C}\left\{\|y_T\|_{p'}^2+\left(\int_{T_1}^T\|f(t,0,0)\|_{p'}+g_1(t)\|z(t)\|_{p'}+g_2(t)\|Z(t)\|_{p'}dt\right)^2\right\}\\
      &\indent+\mathcal{C}\left\{\left(\int_{T_1}^T|g_1(t)|dt\right)^2+\int _{T_1}^T  |g_2(t)|^2dt\right\} \left\|(z(\cdot),Z(\cdot))\right\|^2_{C_\mathbb{A}(T_1,T;L^{p'}(\mathscr{C}))\times \mathcal{H}^{p'}(T_1,T)}.
  \end{aligned}
\end{equation}
Let us choose $T_1\in[0,T)$ such that
\begin{equation}\label{coefficient of contraction map}
 \mathcal{C}\left\{\left(\int _{T_1}^T |g_1(t)|dt\right)^2+ \int_{T_1}^T|g_2(t)|^2dt\right\} \leq \frac{1}{4}.
\end{equation}
Then, by \eqref{estimate of contraction map}, we obtain \eqref{The fixed point}. This shows that the map $\Gamma$ is contractive. Hence, there exists a unique fixed point,  which is a solution to \eqref{BQSDE} on $[T_1,T]$.
Moreover,  from \eqref{x X estimate} and \eqref{coefficient of contraction map}, we obtain that
\begin{equation}\label{finally}
\|(y(\cdot), Y(\cdot))\|_{C_\mathbb{A}(T_1,T;L^{p'}(\mathscr{C}))\times \mathcal{H}^{p'}(T_1,T)} \leq \mathcal{C} \left(\|y_T\|_{p'}+\int_{T_1}^T\|f(t,0,0)\|_{p'}dt\right).
\end{equation}
Repeating the above argument, we obtain the existence of solution to
\eqref{BQSDE}.
Finally, the uniqueness and the estimate \eqref{Estimate-back} follow from \eqref{finally}.
\end{proof}

By Assumption \ref{Assump 1} and \cite[Proposition 3.3]{P.X}, there exists constant $\mathcal{C}$ such that
\begin{equation*}
\begin{aligned}
&\Big\|\left\{D_{x}(t)^*\phi+\left(F^e_x(t)-F^o_x(t)+G_x(t)\right)^*\left(\Phi_e-\Phi_o\right)-L_x(t)\right\}\\
&\indent-\left\{D_{x}(t)^*\widetilde{\phi}+\left(F^e_x(t)-F^o_x(t)+G_x(t)\right)^*\left(\widetilde{\Phi}_e-\widetilde{\Phi}_o\right)-L_x(t)\right\}\Big\|_{p'}\\
=&\left\|D_{x}(t)^*(\phi-\widetilde{\phi})+\left(F^e_x(t)-F^o_x(t)+G_x(t)\right)^*\left(\Phi_e-\Phi_o-\widetilde{\Phi}_e+\widetilde{\Phi}_o\right)\right\|_{p'}\\
\leq&\mathcal{C}\left\|\phi-\widetilde{\phi}\right\|_{p'}+\mathcal{C}\left\|\Phi_e-\widetilde{\Phi}_e\right\|_{p'}+\mathcal{C}\left\|\widetilde{\Phi}_o-\Phi_o\right\|_{p'}\\
\leq&\mathcal{C}\left\|\phi-\widetilde{\phi}\right\|_{p'}+2\mathcal{C}\left\|\Phi-\widetilde{\Phi}\right\|_{p'},
\end{aligned}
\end{equation*}
for any $\phi, \widetilde{\phi}\in L^{p'}(\mathscr{C})$, $\Phi, \widetilde{\Phi}\in \mathcal{H}^{p'}(0,T)$. Hence, by Theorem \ref{the solution to QBSDEs}, \eqref{BSDE-p-even}  has a unique solution in $C_\mathbb{A}(0,T;L^{p'}(\mathscr{C}))\times \mathcal{H}^{p'}(0,T)$.
\begin{rem}
In the forthcoming paper \cite{W.W}, we will study the solution to \eqref{BSDE-P-even} by the relaxed transposition method \cite[Chapter 6]{L.Z-2014} and \cite[Chapter 12.4]{L.Z-2020}.
\end{rem}

\end{document}